\documentclass[10pt,a4paper]{article} 
\usepackage{amsmath,amssymb,amsthm,amsfonts,amscd,euscript,verbatim, t1enc, newlfont, graphicx, cancel}
\usepackage{hyperref}
\usepackage[all]{xy}
\providecommand{\keywords}[1]{\textbf{\textit{Key words and phrases }} #1}
\providecommand{\subjclass}[1]{\textbf{\textit{2020 Mathematics Subject Classification.}} #1}

\theoremstyle{definition}
\newtheorem{theo}{Theorem}[subsection]

\newtheorem{theore}{Theorem}[section]
\newtheorem{pr}[theo]{Proposition}

 \newtheorem{lem}[theo]{Lemma}
 \newtheorem{coro}[theo]{Corollary}
  
\theoremstyle{remark}
\newtheorem{rema}[theo]{Remark}

\newtheorem{rrema}[theore]{Remark}
\newtheorem{ddefi}[theore]{Definition}
\theoremstyle{definition}
\newtheorem{defi}[theo]{Definition}

\numberwithin{equation}{subsection}

\newcommand\cu{\underline{C}}
\newcommand\cuz{{\underline{C}^0}}

\newcommand\wz{w^0}
\newcommand\hwz{\hw^0}

\newcommand\du{\underline{D}}
\newcommand\eu{\underline{E}}
\newcommand\au{\underline{A}}
\newcommand\auz{\underline{A}^0}

\newcommand\bu{\underline{B}}
\newcommand\hu{\underline{H}}

\newcommand\obj{\operatorname{Obj}}

\newcommand\id{\operatorname{id}}
\DeclareMathOperator\adfu{\operatorname{AddFun}}
\DeclareMathOperator\adfur{\operatorname{Fun}_R}

\DeclareMathOperator\kar{\operatorname{Kar}}
 \DeclareMathOperator\ke{\operatorname{Ker}}
 \DeclareMathOperator\cok{\operatorname{Coker}}
\DeclareMathOperator\imm{\operatorname{Im}}
\DeclareMathOperator\co{\operatorname{Cone}}

\DeclareMathOperator\inli{\varinjlim}

\newcommand\hw{{\underline{Hw}}}
\newcommand\hrt{{\underline{Ht}}}

\newcommand\alz{{\aleph_0}}
\newcommand\alo{{\aleph_1}}

\newcommand\lo{\mathcal{B}}
\newcommand\ro{\mathcal{A}}

\newcommand\spe{\operatorname{Spec}}
\newcommand\modd{\operatorname{Mod}}

\newcommand\com{\mathbb{C}}

\newcommand\z{{\mathbb{Z}}}

 \newcommand\lan{\langle}
\newcommand\ra{\rangle}

\newcommand\al{\alpha}
\newcommand\be{\beta}

\newcommand\ns{\{0\}}
\newcommand\ab{\operatorname{Ab}}

\newcommand\cp{\mathcal{P}}
\newcommand\perpp{{}^{\perp}}
\newcommand\opp{^{op}}

\newcommand\tba{Add text!!}

\newcommand\tm{\tilde{M}}

\newcommand\ccu{\underline{\tilde{\mathcal{C}}}} 
\newcommand\wu{\tilde{w}}
\newcommand\hwu{{\underline{H\tilde{w}}}}
\newcommand\mmodd{\operatorname{mod}}

\begin{document}

\title{On  $t$-structures adjacent and orthogonal to weight structures}
\author{Mikhail V. Bondarko
   \thanks{ 
 The work is supported by  the Leader (Leading scientist Math) grant no. 22-7-1-13-1 and by the Ministry of Science and Higher Education of the Russian
Federation (agreement no. 075–15–2025–343).}}\maketitle
\begin{abstract} 
 We study  $t$-structures (on triangulated categories) that are closely related to weight structures.
 A $t$-structure couple $t=(\cu_{t\le 0},\cu_{t\ge 0})$ is said to be  {\it adjacent} to a  weight structure $w=(\cu_{w\le 0}, \cu_{w\ge 0})$ if $\cu_{t\ge 0}=\cu_{w\ge 0}$. 
 For  a triangulated category $\cu$ that satisfies the Brown representability property 
  we prove that $t$ that is  adjacent to $w$ exists if and only if $w$ is {\it smashing} (i.e.,  $\cu_{w\ge 0}$ is $\coprod_{\cu}$-closed); 
     the heart of this $t$ is the category of those functors $\hw\opp\to \ab$ that respect products (here $\hw$ is the heart of $w$). This result has important applications.

    We prove several more statements on constructing ({\it orthogonal}) $t$-structures starting from weight structures. 
    Some of these results 
     generalize properties of semi-orthogonal decompositions proved in 
      an earlier paper, and can be applied to various derived categories of (quasi)coherent sheaves on a scheme $X$ that is projective over an affine noetherian one. We 
 also study hearts of orthogonal $t$-structures and their restrictions, and prove some statements on 'reconstructing' weight structures from orthogonal $t$-structures. 

The main tool of this paper is the notion of virtual $t$-truncations of cohomological functors; these are defined in terms of weight structures and behave as if they come from $t$-truncations of representing objects (whether $t$ exists or not).
\end{abstract}


\tableofcontents
 \section*{Introduction}
This paper is devoted to the study of those $t$-structures that are closely related to weight structures (on various triangulated categories).

Let us recall that  $t$-structures on triangulated categories have 
 been important for homological algebra ever since their introduction in \cite{bbd}. Respectively, their study and construction is an actual and non-trivial question.
 Next, in \cite{konk} and \cite{bws} a rather similar notion of a weight structure $w$ on a triangulated category $\cu$ was introduced.
Moreover, in ibid. a $t$-structure $t=(\cu_{t\le 0}, \cu_{t\ge 0})$  was said to be {\it (left) adjacent} to $w$ if $\cu_{t\ge 0}=\cu_{w\ge 0}$, and certain examples of adjacent structures were constructed. 
Furthermore, in \cite{bger} 
 for a $t$-structure $t$ on a triangulated category $\cu'$ that is related to $\cu$ by means of  a {\it duality} bi-functor 
 a more general notion of {\it 
   orthogonality} of  a weight structure $w$ on $\cu$  to $t$ was introduced. 
 Also, the relationship between the hearts of adjacent and orthogonal structures was studied in detail. 

Next, if $w$ is 
  adjacent to $t$ then it determines $t$ uniquely and vice versa. 
 Yet the only previously existing way of constructing $t$ if $w$ is given was to 
 use certain 'nice generators' of $w$ (see Definition \ref{dcomp}(\ref{dgenw}), 
 Theorem 3.2.2 of \cite{bwcp}, and 
  Theorem 2.4.2(II.2) of \cite{bgroth}). However, already in \cite{bws} the notion of virtual $t$-truncations for (co)homological functors was introduced, and it was proved that virtual $t$-truncations possess several nice properties. In particular, it was demonstrated that these are closely related to $t$-structures (whence the name) even though they are defined in terms of weight structures only. 

In the current paper we propose a new construction method. We prove that adjacent and orthogonal $t$-structures can be constructed using virtual $t$-truncations whenever certain Brown representability-type assumptions on $\cu$ (and $\cu'$) are known.
 Respectively, our results yield the existence of some 
 new families of $t$-structures. 

Let us formulate one of these results. For a triangulated category $\cu$ that is {\it smashing}, i.e., closed with respect to (small) coproducts, and a weight structure $w$ on it we will say that $w$ is {\it smashing} whenever $\cu_{w\ge 0}$ is closed with respect to $\cu$-coproducts (note that $\cu_{w\le 0}$ is $\coprod$-closed automatically).

\begin{theore}[See Theorem \ref{tsmash}]\label{tadjti}

Let $\cu$ be a smashing triangulated category  
 that satisfies the following Brown representability property: any functor $\cu\opp\to \ab$ that respects ($\cu\opp$)-products is representable.\footnote{Equivalently, a functor is representable if and only if it converts $\cu$-coproducts into products of abelian groups.} 

Then for  a weight structure $w$  on $\cu$ there exists a $t$-structure $t$  adjacent to it if and only if $w$ is smashing. Moreover, the heart of $t$ (if $t$ exists) is equivalent to the category of all those additive functors $\hw\opp\to \ab$ that respect products; here $\hw$  is the {\it heart} of $w$. 
\end{theore}
Note here that (smashing) triangulated categories satisfying the Brown representability property are currently 
 really popular  in several areas of mathematics (thanks to the foundational results of A. Neeman, H. Krause and others); in particular, this property holds if either $\cu$ or $\cu\opp$ is   {\it compactly generated}. 
Moreover, it is easy to construct vast families of smashing weight structures on $\cu$ (at least) if $\cu$ is compactly generated; see Remark \ref{rsmashex}(1) below. 

Certainly the dual to Theorem \ref{tadjti} is valid as well. Moreover, if $\cu\opp$ satisfies the dual Brown representability property and $w$ is both cosmashing and smashing then  
 its {\it right adjacent} $t$-structure $t$ (characterized by the assumption $\cu_{t\le 0}=\cu_{w\le 0}$) restricts to the subcategory of compact objects of $\cu$ as well as to all other 'levels of smallness' for objects (see Theorem \ref{tcosm}(3)). Combining this statement with (the existence of weight structures) Theorem 3.1 of \cite{kellerw} we obtain a statement on $t$-structures extending  Theorem 7.1 of ibid.  

We also prove some alternative versions of Theorem \ref{tadjti} that can be applied to 'quite small' triangulated categories.
 Instead of the Brown representability condition for $\cu$ one can demand it to satisfy the {\it $R$-saturatedness} one (see Definition \ref{dsatur}(2) below; this is an  $R$-linear finite analogue of the Brown representability). Then for any {\it bounded} $w$ on $\cu$ (that is, $\cup_{i\in \z}\cu_{w\le 0}[i]=\cup_{i\in \z}\cu_{w\ge 0}[i]=\cu$) there will exist a $t$-structure  adjacent to it. Moreover, 
  this statement also holds for {\it essentially bounded} weight structures (see Corollary \ref{csatur}(1)); those generalize both bounded weight structures and semi-orthogonal decompositions. According to 
a saturatedness statement from \cite{neesat} 
 this result can be applied to the  derived category $D^{perf}(X)$  
of perfect complexes on a regular scheme 
  that is proper over the spectrum of a Noetherian ring $R$ 
   (see 
    Corollary \ref{csatur}(3)). 
 We also study restrictions of  orthogonal $t$-structures to  subcategories corresponding to certain bounds and support conditions; see Propositions \ref{pcgrlin} and \ref{pcgrlintba}, Corollaries \ref{csatur}(2) and \ref{crling}, and Theorem \ref{tcoprb}. We demonstrate that these results can be applied to 
certain  derived categories of (quasi)coherent sheaves; this includes the case of a singular $X$ (that is proper over $\spe R$; see Propositions \ref{pgeomap1} and \ref{pgeomap2}). 
 
 Lastly, we try to prove some converse of these results. So we study the question when a (fixed) $t$-structure $t$ on a triangulated category $\cu$ is  orthogonal  to a weight structure $w$ on a 
 'large enough' subcategory $\cu'\subset\cu$. 
Roughly, this is the case if and only if the heart $\hrt$ has enough projectives and those lift to $\cu$ in a certain way (and the latter condition follows from duals to certain Brown representability-type conditions). However, it appears to be difficult to find the conditions that ensure the existence of an adjacent weight structure (so, 
 $\cu'=\cu$). We prove some statements of this sort; yet these results appear to be 
 more difficult to apply than the aforementioned converse ones. 

\begin{rrema}\label{rneetgeni}

1.  The author 
  does not claim that the methods of the current paper are the most general among the existing methods for constructing $t$-structures. Yet they give more information on the hearts of the corresponding $t$-structures than some other results in this direction; cf. Remark \ref{rneetgen} below for more detail.

2. The current text is a certain modification of the preprint \cite{bvtr}. We will say more on the comparison between these two texts in Remark \ref{rbvtt} below.

3. The 
work on \cite{bvtr} inspired the author to write 
  \cite{bdec}. We note that semi-orthogonal decompositions (that are the main subject of ibid.) are particular (and very important!) cases of weight structures; see 
   Proposition \ref{psod}  below for more detail. 
 Next, both in \cite{bdec} and in 
\S\ref{sebwt} below the 
 main idea is to extend a weight structure to 
  a bigger (smashing) category and to  pass to an orthogonal $t$-structure (see Definitions \ref{dwso}(\ref{idrest}) and \ref{dort} below). It is no wonder that applying these methods in the setting of general weight structures is somewhat more difficult than doing this for semi-orthogonal decompositions. On the other hand, these difficulties make the current paper somewhat more interesting and also 'more original'.    
\end{rrema}

Let us  now describe the contents  of the paper. Some more information of this sort may be found in the beginnings of sections. 

In \S\ref{swt} we give some 
  definitions and conventions, and recall some basics on $t$-structures, weight structures, and semi-orthogonal decompositions. The most original statement of this section is Proposition \ref{psod} that states that  semi-orthogonal decompositions are precisely those weight structures whose heart is zero; these weight structures are also characterized by the uniqueness of 
	 weight decompositions of objects.

In \S\ref{sortvtt} we 
 define  virtual $t$-truncations  and weight range of functors and prove several nice properties for them. We also relate the existence of orthogonal $t$-structures to virtual $t$-truncations; this gives general if and only if criteria for the existence of adjacent $t$-structures. 
 Moreover, we define and study coproductive extensions of weight structures.

In  \S\ref{smashb} we prove Theorem \ref{tadjti}
 and some related statements. So we study smashing triangulated categories along with the existence of $t$-structures adjacent to weight structures on them and certain restrictions of these $t$-structures. We also consider weight structures extended from subcategories of compact objects.

In \S\ref{ststreb} 
we discuss certain bounded 
 analogues of the results of the previous section. 
  In particular, we study the existence of adjacent and orthogonal $t$-structures in $R$-saturated triangulated categories and the restriction of orthogonal $t$-structures to  subcategories corresponding to certain bounds and support conditions. 
We apply our general statements to  various derived categories of (quasi)coherent sheaves. To formulate our statements in a more general 
 setting we recall the notions of essentially bounded (above, below, or both) objects and weight structures. 
  This enables us to generalize some central results of \cite{bdec}.

In \S\ref{sortw} we 
 address the question 
  which $t$-structures are adjacent to weight structures. In certain situations we prove that a weight structure adjacent to a given $t$ exists whenever $\hrt$ has enough projectives and certain additional assumptions are fulfilled.

\S\ref{smore} contains some additional remarks that relate this paper with 
 the literature. None of them appear to be really important.

The author is deeply grateful to prof. A. Neeman for calling his attention to \cite{kellerw} as well as for writing his extremely interesting texts that are crucial for the current paper.

\section{A reminder on  $t$-structures  and weight structures}\label{swt}

In this section we recall the notions of $t$-structures and weight structures, along with orthogonality and adjacency for them. 

In \S\ref{snotata}  we introduce some categorical notation and recall some basics on $t$-structures. 

In \S\ref{sws} we recall some of the theory of weight structures and semi-orthogonal decompositions. In particular,  Proposition \ref{psod}  states that  semi-orthogonal decompositions are precisely those weight structures whose heart is zero.

In \S\ref{sadj} we recall the definitions of adjacent and orthogonal weight and $t$-structures that are central for this paper.

\subsection{Some categorical and $t$-structure notation}\label{snotata}

\begin{itemize}

\item All products and coproducts in this paper will be small.

\item Given a category $C$ and  $X,Y\in\obj C$  we will write
$C(X,Y)$ for  the set of morphisms from $X$ to $Y$ in $C$.

\item All subcategories in this paper will be strictly full. Respectively, for categories $C',C$ we write $C'\subset C$ if $C'$ is a strictly full 
subcategory of $C$.

\item For $S_1,S_2\subset C$ we will write  $S_1\cap S_2$ for the (strictly full) subcategory  of $C$ whose object class equals $\obj S_1\cap\obj S_2$.

\item Given a category $C$ and  $X,Y\in\obj C$, we say that $X$ is a {\it
retract} of $Y$ 
 if $\id_X$ can be 
 factored through $Y$.

\item  Given an additive 
   subcategory $\hu$ of an additive 
     $C$, the (strictly full) subcategory $\kar_C(\hu)$ of $C$  
     whose objects are all $C$-retracts of objects of $\hu$ will be called {\it the retraction-closure} 
       of $\hu$ in $C$.

$\hu$ is said to be {\it retraction-closed} in $C$ if it coincides with  $\kar_C(\hu)$. 

\item  We will say that $C$ is {\it Karoubian} if any idempotent morphism yields a direct sum decomposition in it. 


\item The symbol $\cu$ below will always denote some triangulated category;  it will often be endowed with a weight structure $w$. The symbols $\cu'$, $\du$, and $\cuz$ 
 are also reserved for triangulated categories only. Moreover, we will often assume $\cu,\cu'\subset \du$ and $\cuz\subset \cu$.

\item For any  $A,B,C \in \obj\cu$ we will say that $C$ is an {\it extension} of $B$ by $A$ if there exists a distinguished triangle $A \to C \to B \to A[1]$.

\item For any $D,E\subset \obj \cu$ we will write $D\star E$  for the class of all extensions of elements of $E$ by elements of $D$. 

\item A class $\cp\subset \obj \cu$ is said to be  {\it extension-closed}
    if $0\in \cp$ and $\cp\star \cp\subset \cp$. 




\item The smallest additive retraction-closed extension-closed class of objects of $\cu$ containing   $\cp$  will be called the {\it 
envelope} of $\cp$.  

We will say that  $\cp$ {\it densely generates} $\cu$ if the envelope of $\cup_{i\in \z}\cp[i]$ equals $\cu$.

\item For $X,Y\in \obj \cu$ we will write $X\perp Y$ if $\cu(X,Y)=\ns$. 

For
$D,E\subset \obj \cu$ we write $D\perp E$ if $X\perp Y$ for all $X\in D,\
Y\in E$.

Given $D\subset\obj \cu$ we  will write $D^\perp$ for the class
$$\{Y\in \obj \cu:\ X\perp Y\ \forall X\in D\}.$$
Dually, ${}^\perp{}D$ is the class
$\{Y\in \obj \cu:\ Y\perp X\ \forall X\in D\}$.

\item Given $f\in\cu (X,Y)$, where $X,Y\in\obj\cu$, we will call the third vertex
of (any) distinguished triangle $X\stackrel{f}{\to}Y\to Z$ a {\it cone} of
$f$.

\item Below the symbols $\au$, $\auz$, and $\au'$ will always  denote some abelian categories; $\bu$ is an additive category.

\item 
All complexes in this paper will be cohomological.


	\item We will say that an additive covariant (resp. contravariant) functor from $\cu$ into $\au$ is {\it homological} (resp. {\it cohomological}) if it converts distinguished triangles into long exact sequences.
	
	For a (co)homological functor $H$ and $i\in\z$ we will write $H_i$ (resp. $H^i$) for the composition $H\circ [-i]$. 

\end{itemize}




Let us also recall the notion of  a $t$-structure (mainly to fix  notation). 

\begin{defi}\label{dtstr}

A couple of subclasses  $\cu_{t\le 0},\cu_{t\ge 0}\subset\obj \cu$ will be said to be a
$t$-structure $t$ on $\cu$  if 
they satisfy the following conditions:

(i) $\cu_{t\le 0}$ and $\cu_{t\ge 0}$  are strict, i.e., contain all
objects of $\cu$ isomorphic to their elements.

(ii) $\cu_{t\le 0}\subset \cu_{t\le 0}[1]$ and $\cu_{t\ge 0}[1]\subset \cu_{t\ge 0}$.

(iii)  $\cu_{t\ge 0}[1]\perp \cu_{t\le 0}$.

(iv) For any $M\in\obj \cu$ there exists a  {\it $t$-decomposition} distinguished triangle
\begin{equation}\label{tdec}
L_tM\to M\to R_tM{\to} L_tM[1]
\end{equation} such that $L_tM\in \cu_{t\ge 0}, R_tM\in \cu_{t\le 0}[-1]$.

\end{defi}

We will also give some auxiliary definitions.

\begin{defi}\label{dtstro}
1. For any $i\in \z$ we will use the notation $\cu_{t\le i}$ (resp. $\cu_{t\ge i}$) for the class $\cu_{t\le 0}[i]$ (resp. $\cu_{t\ge 0}[i]$). 

2. $\hrt$ is the (full) subcategory of $\cu$ whose object class is $\cu_{t=0}=\cu_{t\le 0}\cap \cu_{t\ge 0}$.

3. We will say that $t$ is {\it bounded below} if $\cup_{i\in \z}\cu_{t\ge i}=\obj\cu$.

Moreover, we   say that $t$ is {\it bounded} if the equality $\cup_{i\in \z}\cu_{t\le i}=\obj\cu$ is valid as well.

4. Let $\cuz$ be a full triangulated subcategory of $\cu$.

We will say that $t$ {\it restricts} to $\cuz$ whenever the couple $t_{0}= (\cu_{t\le 0}\cap \obj \cuz,\ \cu_{t\ge 0}\cap \obj \cuz)$ is a $t$-structure on $\cuz$.
\end{defi}

Let us recall a few properties of $t$-structures.

\begin{pr}\label{prtst}
Let $t$ be a $t$-structure on a triangulated category $\cu$. Then the following statements are valid.

\begin{enumerate}


\item\label{itcan}
The triangle (\ref{tdec}) is canonically and functorially determined by $M$. Moreover, $L_t$ is right adjoint to the embedding $ \cu_{t\ge 0}\to \cu$ (if we consider $ \cu_{t\ge 0}$ as a full subcategory of $\cu$) and $R_t$ is left adjoint to  the embedding $ \cu_{t\le -1}\to \cu$.

\item\label{itha}
$\hrt$ is 
  an abelian category with short exact sequences corresponding to distinguished triangles in $\cu$.

\item\label{itho}
For any $n\in \z$ we will use the notation $t_{\ge n}$ for the functor $[n]\circ L_t\circ [-n]$, and $t_{\le n}=[n+1]\circ R_t\circ [-n-1]$.

Then there is a canonical isomorphism of functors $t_{\le 0}\circ t_{\ge 0}\cong t_{\ge 0}\circ  t_{\le 0}$ 
  (if we consider these functors as endofunctors of $\cu$), and the composite functor $H^t=H_0^t$ actually takes values in the subcategory $\hrt$ of $ \cu$. Furthermore, this functor $H^t:\cu \to \hrt$   is homological.

\item\label{itperp} $\cu_{t\le 0}= \cu_{t\ge 1}\perpp$ and $\cu_{t\ge 0}=\perpp(\cu_{t\le- 1})$. 

Consequently, these classes are retraction-closed and extension-closed in $\cu$.

Thus $t$ is uniquely determined both by $\cu_{t\ge 0}$ and by $\cu_{t\le 0}$.

\item\label{itcopr}
$\cu_{t\ge 0}$ is closed with respect to   all (small) coproducts that exist in $\cu$.
\end{enumerate}
\end{pr}
\begin{proof}
 Assertions \ref{itcan}--\ref{itperp}  were essentially established in \S1.3 of \cite{bbd} (yet see Remark \ref{rtst}(2) below).
 
 Assertion \ref{itcopr} immediately follows from assertion \ref{itperp}.
 \end{proof}

\begin{rema}\label{rtst}

1. The notion of a $t$-structure is clearly self-dual, that is, the couple $t\opp=(\cu_{t\ge 0},\cu_{t\le 0})$ gives a $t$-structure on the category $\cu\opp$. We will say that the latter $t$-structure is {\it opposite} to $t$.   

2. Even though in \cite{bbd} where $t$-structures were introduced  and in several preceding papers of the author the so-called cohomological convention for $t$-structures was used, in the current text we 
 use the homological convention; the reason for this is that it is coherent with the homological convention for weight structures (see Remark \ref{rstws}(3) below). Respectively, 
 our notation $\cu_{t\ge 0}$ 
 corresponds to the class $\cu^{t\le 0}$ in the cohomological convention. \end{rema}

We will also need two simple statements on restrictions of $t$-structure.

\begin{lem}\label{lrest}
Assume that $\cu$ is endowed with a $t$-structure $t$, and $\cu^0,\cu^1$ are triangulated subcategories of $\cu$.

1. $t$ restricts to $\cuz$ if and only if the functor $L_t$ sends $\cuz$ into itself. 

2. Take $\cu^2=\cuz\cap \cu^1$ (that is, the full subcategory of $\cu$ whose object class equals $\obj \cu^1\cap\obj \cu^2$). Then $\cu^2$ is triangulated.

Moreover, if $t$ restricts both to $\cu^0$ and $\cu^1$ then it restricts to $\cu^2$ as well, and the heart of this restriction equals the intersections of the hearts for $\cu^0$ and $\cu^1$.
\end{lem}
\begin{proof}

1. If $t$ restricts to $\cuz$ then the uniqueness of the triangles  (\ref{tdec}) (see Proposition \ref{prtst}(\ref{itcan})) implies that $L_t$ sends $\cuz$ into itself indeed.

Conversely, to check whether $t$ restricts to $\cuz$ it clearly suffices to verify that any object $M$ of $\cuz$ possesses a decomposition of the sort (\ref{tdec}) inside $\cuz$. Now, if $L_t(M)\in \obj \cuz$ then $R_t(M)$ is an object of $\cuz$ as well, since $\cuz$ is a triangulated subcategory of $\cu$.

2. $\cu^2$ is triangulated since $\cu^0$ and $\cu^1$ are (fully) strict triangulated subcategories of $\cu$, whereas 
different choices of cones of a $\cu$-morphisms are connected by (non-unique) isomorphisms. 

The 
 existence of the restricted $t$-structure statement  is easy as well; one can immediately deduce this statement from assertion 1. Lastly, the heart assertion is an immediate consequence of our definitions. 
 \end{proof}

\begin{rema}\label{rtrest}
One may easily add some more equivalent conditions to Lemma \ref{lrest}(1); cf. Theorem \ref{trefl}(I) below.
\end{rema}

\subsection{On weight structures and semi-orthogonal decompositions}\label{sws}

Let us recall some basic  definitions of the theory of weight structures. 

\begin{defi}\label{dwstr}

I. A pair of subclasses $\cu_{w\le 0},\cu_{w\ge 0}\subset\obj \cu$ 
will be said to define a weight
structure $w$ on a triangulated category  $\cu$ if 
they  satisfy the following conditions.

(i) $\cu_{w\le 0}$ and $\cu_{w\ge 0}$ are 
retraction-closed in $\cu$ (i.e., contain all $\cu$-retracts of their objects).

(ii) {\bf Semi-invariance with respect to translations.}

$\cu_{w\le 0}\subset \cu_{w\le 0}[1]$, $\cu_{w\ge 0}[1]\subset
\cu_{w\ge 0}$.

(iii) {\bf Orthogonality.}

$\cu_{w\le 0}\perp \cu_{w\ge 0}[1]$.

(iv) {\bf Weight decompositions}.

 For any $M\in\obj \cu$ there
exists a distinguished triangle
\begin{equation}\label{ewd}
L_wM\to M\to R_wM {\to} L_wM[1]
\end{equation}
such that $L_wM\in \cu_{w\le 0} $ and $ R_wM\in \cu_{w\ge 0}[1]$.
\end{defi}

We will also need the following definitions.

\begin{defi}\label{dwso}
Let $i,j\in \z$; assume that a triangulated category $\cu$ is endowed with a weight structure $w$.

\begin{enumerate}
\item\label{idh}
The full  subcategory $\hw$ of $ \cu$ whose objects are
$\cu_{w=0}=\cu_{w\ge 0}\cap \cu_{w\le 0}$ 
 is called the {\it heart} of 
$w$.

\item\label{id=i}
 $\cu_{w\ge i}$ (resp. $\cu_{w\le i}$,  $\cu_{w= i}$) will denote the class $\cu_{w\ge 0}[i]$ (resp. $\cu_{w\le 0}[i]$,  $\cu_{w= 0}[i]$).

\item\label{id[ij]}
$\cu_{[i,j]}$  denotes $\cu_{w\ge i}\cap \cu_{w\le j}$.

\item\label{idrest}
Let $\cuz$ be a (strictly full) triangulated subcategory of $\cu$.

We will say that $w$ {\it restricts} to $\cuz$ whenever the couple $w_{0}= (\cu_{w\le 0}\cap \obj \cuz,\ \cu_{w\ge 0}\cap \obj \cuz)$ is a weight structure on $\cuz$.

Moreover, in this case we will also say that $w$ is an {\it extension} of $w_{0}$ (to $\cu$).



\item\label{idbob} We will call $\cup_{i\in \z} \cu_{w\ge i}$ (resp. $\cup_{i\in \z} \cu_{w\le i}$) the class of {\it $w$-bounded below} (resp., {\it $w$-bounded above}) objects of $\cu$.

Moreover, we say that $w$ is {\it bounded below} (resp. {\it bounded above}, resp. {\it bounded}) if all objects of $\cu$ are bounded below (resp.  bounded above, resp. are bounded both below and above).

 \item\label{idwe}
Let  $\cu'$ be a triangulated category endowed with a weight structure $w'$; 
let $F:\cu\to \cu'$ be an exact functor.

We will say that $F$ is {\it weight-exact} (with respect to $(w,w')$) if $F(\cu_{w\le 0})\subset \cu'_{w'\le 0}$ and  $F(\cu_{w\ge 0})\subset \cu'_{w'\ge 0}$.
 
 \end{enumerate}
\end{defi}

\begin{rema}\label{rstws}

1. A weight decomposition (of any $M\in \obj\cu$) is almost never canonical 
 (see 
  Proposition \ref{psod} below for more detail). 

Still for any $m\in \z$ the axiom (iv) gives the existence of a distinguished triangle \begin{equation}\label{ewdm} w_{\le m}M\to M\to w_{\ge m+1}M\to (w_{\le m}M)[1] \end{equation}  
with some $ w_{\le m}M\in \cu_{w\le m}$   and $ w_{\ge m+1}M\in \cu_{w\ge m+1}$; we will call it an {\it $m$-weight decomposition} of $M$.

 We will often use this notation below (even though $w_{\ge m+1}M$ and $ w_{\le m}M$ are not canonically determined by $M$); we will call any possible choice either of $w_{\ge m+1}M$ or of $ w_{\le m}M$ (for any $m\in \z$) a {\it weight truncation} of $M$. Moreover, when we will write arrows of the type $w_{\le m}M\to M$ or $M\to w_{\ge m+1}M$ we will always assume that they come from some $m$-weight decomposition of $M$.

2. In the current paper (along with several previous ones) we use the ``homological convention'' for weight structures, 
whereas in \cite{bws} and \cite{bger}, 
 the cohomological convention was used. In the latter convention 
the roles of $\cu_{w\le 0}$ and $\cu_{w\ge 0}$ are interchanged, i.e., one takes   $\cu^{w\le 0}=\cu_{w\ge 0}$ and $\cu^{w\ge 0}=\cu_{w\le 0}$. 
 
We also recall that 
   weight structures were independently introduced in \cite{konk};  D. Pauksztello has called them co-t-structures. 
\end{rema}


\begin{pr}\label{pbw}
Let  
$m\le n\in\z$, $M,M'\in \obj \cu$, $g\in \cu(M,M')$. 

\begin{enumerate}
\item \label{idual}
The axiomatics of weight structures is self-dual, i.e., on $\cu'=\cu^{op}$
(so $\obj\cu'=\obj\cu$) there exists the (opposite)  weight structure $w'$ for which $\cu'_{w'\le 0}=\cu_{w\ge 0}$ and $\cu'_{w'\ge 0}=\cu_{w\le 0}$.

\item\label{iort}
 $\cu_{w\ge 0}=(\cu_{w\le -1})^{\perp}$ and $\cu_{w\le 0}={}^{\perp} \cu_{w\ge 1}$.

\item\label{icoprod} $\cu_{w\le 0}$ is closed with respect to all 
  coproducts that exist in $\cu$.


\item\label{icompl} 
				For any (fixed) $m$-weight decomposition of $M$ and an $n$-weight decomposition of $M'$  (see Remark \ref{rstws}(2))
 $g$ can be extended 
to a 
morphism of the corresponding distinguished triangles:
 \begin{equation}\label{ecompl} \begin{CD} w_{\le m} M@>{c}>>
M@>{}>> w_{\ge m+1}M\\
@VV{h}V@VV{g}V@ VV{j}V \\
w_{\le n} M'@>{}>>
M'@>{}>> w_{\ge n+1}M' \end{CD}
\end{equation}

Moreover, if $m<n$ then this extension is unique (provided that the rows are fixed).

\item\label{isplit} If $A\to B\to C\to A[1]$ is a $\cu$-distinguished triangle 
 and $A,C\in  \cu_{w=0}$ then this distinguished triangle splits; hence $B\cong A\bigoplus C\in \cu_{w=0}$.

\item\label{iwdmod} If $M$ belongs to $ \cu_{w\le 0}$ (resp. to $\cu_{w\ge 0}$) then it is a retract of any choice of $w_{\le 0}M$ (resp. of $w_{\ge 0}M$).

\item\label{iwd0} 
 If $M\in \cu_{w\ge m}$ 
 then $w_{\le n}M\in \cu_{[m,n]}$ (for any $n$-weight decomposition of $M$).


	
\item\label{iuni} Let  $v$ be another weight structure for $\cu$; assume   $\cu_{w\le 0}\subset \cu_{v\le 0}$ and $\cu_{w\ge 0}\subset \cu_{v\ge 0}$.\footnote{I.e., the identity on $\cu$ is weight-exact with respect to $(w,v)$.}       
  Then $w=v$ (i.e., these inclusions are equalities).
\end{enumerate}
\end{pr}
\begin{proof}
 All these statements 
 were essentially proved in \cite{bws} (yet pay attention to Remark \ref{rstws}(3) above!). 
 \end{proof}
 
 Now we pass to couples that are simultaneously weight structures and $t$-structures.
 
 \begin{defi}\label{ddec}
Assume that $\ro$ and $\lo$ are strictly full triangulated subcategories of $\cu$.

Then  the couple $D=(\ro,\lo)$ 
 is a 
   {\it semi-orthogonal decomposition} of $\cu$ (or just gives a decomposition of $\cu$) 
  if $\obj \lo\perp \obj \ro$ and $\obj \lo \star \obj \ro =\obj\cu$.\footnote{Recall that  $\obj \lo \star \obj \ro$  is the class of all extensions of objects of $\lo$ by objects of $\ro$.}
 \end{defi}
 
 Below we will need some well-known properties of semi-orthogonal decompositions.
 
 \begin{pr}\label{psod} 
  The following assumptions on a couple $(C_1,C_2)$ 
  of subclasses of $\obj\cu$ are equivalent. 
  
 \begin{enumerate}
 
 \item\label{isodd}
  The full subcategories of $\cu$ corresponding to $C_2$ and $C_1$ give a semi-orthogonal decomposition of $\cu$.
 
  \item\label{isodt}  $(C_1,C_2)$ is a $t$-structure and $C_1[1]=C_1$.
  
   \item\label{isoda} The full subcategory $\cu_1$ of $\cu$ corresponding to $C_1$ is strict, triangulated, and {\it right admissible}, that is, there exists a right adjoint to the embedding $\cu_1\to \cu$; $C_2=C_1\perpp$. 
  
  \item\label{isodw}  $(C_2,C_1)$ 
   is a weight structure and $C_1[1]=C_1$.
   
   \item\label{isodwp}  $(C_2,C_1)$  is a weight structure, $C_1[1]=C_1$,  and  $C_2[1]=C_2$.
 
  \item\label{isodwhz} $(C_2,C_1)$ 
   is a weight structure $w$ and $\hw=\ns$.
   
    \item\label{isodwcand} $(C_2,C_1)$ 
   is a weight structure $w$ and the triangle (\ref{ewd}) is functorially determined by $M$.
   
   \item\label{isodwcann} $(C_2,C_1)$ 
   is a weight structure $w$ and the triangle (\ref{ewd}) is 
    determined by  $M$   up to a non-canonical isomorphism. 
   
 
 \end{enumerate}
\end{pr}
\begin{proof}
 The equivalence of conditions \ref{isoda}, \ref{isodt}, and \ref{isodw} can be easily obtained by means of combining
 Propositions 3.4(4) and 
  3.2(1,2) of \cite{bvt}  (pay attention to Remark \ref{rtst}(2) above!). 
  
  The equivalence of conditions \ref{isodd} and \ref{isoda} is very well-known; see Proposition 2.5 of \cite{bdec}.
  
  Next, condition \ref{isodwp} clearly implies \ref{isodw}, and applying Proposition \ref{pbw}(\ref{iort}) (or Proposition  3.2(2) of \cite{bvt}) we obtain the converse implication.
  
  Now, if  condition \ref{isodwp} 
    is fulfilled then $\cu_{w=0}=\cu_{w=-1}\perp \cu_{w=0}$; thus we obtain condition \ref{isodwhz}. Conversely, if condition \ref{isodwhz} is fulfilled then  Proposition \ref{pbw}(\ref{iwd0}) easily implies that $C_2\subset C_2[1]$. Thus we can apply Proposition \ref{pbw}(\ref{iort}) once again to obtain condition \ref{isodw}.
    
    Next, if $0\neq C\in \cu_{w=0}$ then the triangles $0\to 0\to 0\to 0$ and $C\to 0\to C[1]\to C[1]$ clearly give two non-isomorphic weight decompositions of $0$. Hence condition \ref{isodwcann} implies condition \ref{isodwhz}.
    
 Clearly,    condition \ref{isodwcand} implies condition  \ref{isodwcann}.
 
     Lastly, assume that condition \ref{isodwp} is fulfilled. Then any weight decomposition triangle is clearly a  $1$-weight decomposition  triangle as well. Consequently, the uniqueness provided by  Proposition \ref{pbw}(\ref{icompl}) yields condition \ref{isodwcand}; cf. Remark 1.2.6 of \cite{bkw} where compositions of the corresponding morphisms between decomposition triangles is discussed.
 \end{proof}

\begin{rema}\label{rsod} 
1. Roughly, semi-orthogonal decompositions are (essentially) the weight structures that are also $t$-structures. One may also characterize 
   semi-orthogonal decompositions  both as shift-stable weight structures and as shift-stable $t$-structures.

2. 
 This  notion is obviously self-dual (as well). That is, permuting the classes in the couple yields a semi-orthogonal decomposition in the category $\cu\opp$; cf. Remark \ref{rtst}(1) and Proposition \ref{pbw}(\ref{idual}).

3. Below we will also need the following simple criterion: if $\cu$ is smashing (cf. Theorem \ref{tadjti}) then a semi-orthogonal decomposition $w$ is smashing if and only if the right adjoint functor provided by condition \ref{isoda} of Proposition \ref{psod} respects coproducts; see Proposition 3.4(5)  of \cite{bvt}.
\end{rema}

\subsection{On orthogonal  and adjacent  structures}\label{sadj}

Now let us give a certain definition of orthogonality for weight and $t$-structures (cf.  Remark \ref{rort}(\ref{irort1}) below). 

\begin{defi}\label{dort}
Assume that $\cu$ and $\cu'$ are (full) triangulated subcategories of a triangulated category $\du$, $w$ is a weight structure on $\cu$ and $t$ is a $t$-structure on $\cu'$.

1. 
 We will say that $w$ and $t$ are {\it orthogonal} (in $\du$)  
  if $\cu_{w\le 0}\perp_{\du} \cu'_{t\ge 1}$ and $\cu_{w\ge 0}\perp_{\du} \cu'_{t\le -1}$.

Dually,  we  say that $w$ and $t$ are {\it anti-orthogonal} 
whenever
$ \cu'_{t\ge 1} \perp_{\du} \cu_{w\le 0}$ and $\cu'_{t\le -1} \perp_{\du} \cu_{w\ge 0} $.

2. If $\cu=\cu'=\du$ and $w$ is 
 (anti) orthogonal to  $t$ 
then we will also say that $w$ and $t$ are left (resp. right) {\it adjacent}.

3. We will say that $t$ is  {\it strictly  orthogonal} to $w$ 
  if $\cu'_{t\ge 1}=\cu_{w\le 0}^{\perp_{\du}}\cap \obj \cu'$ and $\cu'_{t\le -1}=  \cu_{w\ge 0}^{\perp_{\du}} \cap \obj \cu'$.
\end{defi}

\begin{rema}\label{rplr}

 Our main statements 
below yield $t$-structures that are strictly orthogonal to the corresponding weight structures. 

Note however that strictness is not automatic. In particular, it is clearly not fulfilled if $\cu\perp \cu'$ and $\cu'\neq 0$. Indeed, then arbitrary $w$ and $t$ would be orthogonal, whereas $\cu_{w\le 0}^{\perp_{\du}}\cap \obj \cu'= \cu_{w\ge 0}^{\perp_{\du}} \cap \obj \cu'=\obj \cu'$. 
\end{rema}

Let us now relate 
 the notion of adjacent structures to the definition given  in \cite{bws}.

\begin{pr}\label{portadj}
For $\cu$, $w$, and $t$ as in Definition \ref{dort}(2) we have the following: $w$ and $t$ are (left) adjacent if and only if $\cu_{w\ge 0}=\cu_{t\ge 0}$.
\end{pr}
\begin{proof}
If $\cu_{w\ge 0}=\cu_{t\ge 0}$ then $w$ and $t$ are (left) adjacent  immediately from the orthogonality axioms of weight and $t$-structures (see Definition \ref{dtstr}(iii) and Definition \ref{dwstr}(iii)). Conversely, if $w$ is  adjacent to $t$ then combining the orthogonality conditions with Proposition \ref{pbw}(\ref{iort}) and Proposition \ref{prtst}(\ref{itperp})  we obtain that
$\cu_{t\ge 0}\subset \cu_{w\ge 0}$ and $\cu_{w\ge 0}\subset \cu_{t\ge 0}$. Hence $\cu_{w\ge 0}= \cu_{t\ge 0}$ as desired.
\end{proof}

\begin{rema}\label{rort}
Proposition \ref{portadj} says that our definition of adjacent structures is essentially equivalent to the original Definition  4.4.1 of \cite{bws}.

\end{rema}

We also make a simple observation concerning semi-orthogonal decomposition.

\begin{pr}\label{portsod}

Assume that $w$ and $t$ satisfy the assumptions of Definition \ref{dort}, $t$ is   strictly  orthogonal to $w$ 
(note that this condition follows from orthogonality whenever $\cu'=\cu$; see Theorem \ref{trefl}(I) below), and the couple $(\cu_{w\ge 0},\cu_{w\le 0})$ is a semi-orthogonal decomposition (see  Definition \ref{ddec}). Then  $(\cu'_{t\le 0},\cu'_{t\ge 0})$ is a semi-orthogonal decomposition as well.

\end{pr}
\begin{proof}
Applying  Proposition \ref{psod} (see conditions 
\ref{isodw} and \ref{isodd} in it) along with the definition of strict orthogonality we obtain that $\cu'_{t\ge 0}=\cu'_{t\ge 0}[1]$. Hence $(\cu'_{t\le 0},\cu'_{t\ge 0})$ is a semi-orthogonal decomposition indeed; see condition \ref{isodt} in  Proposition \ref{psod}. 
\end{proof}

\section{
On virtual $t$-truncations and their relation to orthogonal $t$-structures}\label{sortvtt} 

This section is  devoted to virtual $t$-truncations of functors (these come from weight structures) and their 
 relationship 
with orthogonal $t$-structures.

In \S\ref{svtt} we recall the definition of virtual $t$-truncations of 
 functors and establish their main general properties.

In \S\ref{swr} we define  
  weight range of functors and relate it to virtual $t$-truncations.
	
	In \S\ref{srlin} we prove a few easy properties of $R$-linear categories and relate them to virtual $t$-truncations. 


In \S\ref{sorts} we prove that if $\cu'\subset \cu$ then for a weight structure $w$ on $\cu$ there exists at most one $t$-structure on $\cu'$ (strictly) orthogonal to it; see Theorem \ref{trefl}. Moreover, if $\cuz$ 
 is $R$-linear, $\wz$ is a weight structure on $\cuz$, and $t$ is a $t$-structure on $\cu$ strictly orthogonal to $\wz$ then for any weak Serre subcategory $\au$ of $R-\modd$ one can specify a subcategory $\cu_{\au}$ such that $t$ restricts to it.
 

In \S\ref{scopra} we define {\it coproductive extensions} of weight structures (from subcategories).  For $(w,t,\cu,\cu')$ as above we prove that if $w$ is  a coproductive extension of a   weight structure $\wz$ on $\cuz\subset \cu$ then $w$ is strictly ortogonal to $t$ if and only if $\wz$ is (see Corollary \ref{ccopr}); hence $\wz$ determines $t$ uniquely. 

\subsection{Virtual $t$-truncations 
 of (co)homological functors}\label{svtt} 


Let us  give some definitions; some of them are central for this paper.

\begin{defi}\label{dvtt}

Assume that $\cu$ is endowed with a weight structure $w$, $ n\in \z$, and $\au$ is  an abelian category.

1. Let $H$ be a 
 contravariant functor from $\cu$ into $\au$. 

 We define the {\it virtual $t$-truncation} functor $\tau_{\le n }(H)$ (resp. $\tau_{\ge n }(H)$)  by the correspondence $$M\mapsto\imm (H(w_{\le n+1}M)\to H(w_{\le n}M)) ;$$ 
(resp. $M\mapsto\imm (H(w_{\ge n}M)\to H(w_{\ge n-1}M)) $). Here we take arbitrary choices of  
the corresponding weight truncations of $M$ and connect them by means of applying Proposition \ref{pbw}(\ref{icompl}) in the case $g=\id_M$.

2. Dually, let $H':\cu\to \au$ be a 
 covariant functor.  Then we will write $\tau_{\le n }(H')$ for the correspondence $M\mapsto\imm (H'(w_{\le n}M)\to  H'(w_{\le n+1}M)) $ and $\tau_{\ge n }(H')=M\mapsto\imm (H'(w_{\ge n-1}M)\to H'(w_{\ge n}M)) $ (here we take the same connecting arrows between weight truncations of $M$ as we did above).

	3.  Assume that $\cu$ is a triangulated 
	subcategory of a triangulated category $\du$. Then for any $M\in \obj \du$ we will write $H_M=H_{M}^{\cu}$  (resp. $H^M=H^M_{\cu}$) for the restriction of the functor (co)represented by $M$ to $\cu$ (thus $H_M$ and $H^M$ are functors from $\cu$ into $\ab$). 
 Moreover, sometimes we will say that 
  $H_M$ (resp. $H^M$) is $\du$-(co)represented by $M$.
 
 Furthermore, for a $\du$-morphism $f:M_1\to M_2$ we will write $H_f^{\cu}$ for the transformation  $H_{M_1}^{\cu}\to H_{M_2}^{\cu}$ corresponding to $f$. 

\end{defi}

\begin{rema}\label{rnotcoh}
In the current paper we only need virtual $t$-truncations of (co)homological functors.

Respectively, starting from \S\ref{swr} we will never consider virtual $t$-truncations of functors that are neither cohomological nor homological.

\end{rema}

Now we recall the main properties of  virtual $t$-truncations. 

\begin{pr}\label{pwfil}
In the notation of the previous definition the  following statements are valid.
\begin{enumerate}
\item\label{iwfilp1} 
The objects $\tau_{\le n}(H)(M)$ and $\tau_{\ge n}(H)(M)$ are $\cu$-functorial  in $M$  (and essentially do not depend on any choices).

\item\label{iwfilp2} 
 The functors $\tau_{\le n}(H)$ and $\tau_{\ge n}(H)$ are cohomological if $H$ is.

\item\label{iwfilp3} 
 If $H$ is cohomological then there exists 
 a long exact sequence 
\begin{equation}\label{evtt}
\begin{gathered} 
\dots \to \tau_{\le n-1}(H)\circ [-1] \to \tau_{\ge n }(H)\stackrel{a_n}{\to} H \\ \stackrel{b_{n-1}}{\to} \tau_{\le n -1}(H)\to \tau_{\ge n }(H)\circ [1]\to H
\circ [1]\to\dots\end{gathered} 
 \end{equation} of functors (i.e.,  the result of applying this sequence to any object of $\cu$ is a long exact sequence).  The shift of this exact sequence by $3$ positions to the right  is given by composing the functors with $[1]$. 

\item\label{iwfilp4} 
 Assume that there exists a $t$-structure $t$ on $\cu'$  that is orthogonal 
   to $w$ (for certain $\cu'$ and $\du$ as in Definition \ref{dort}). 
Then for any $M\in \obj \cu'$ 
 the functors $ \tau_{\ge n }(H_M)$ and $\tau_{\le n}(H_M)$  are $\du$-represented (on $\cu$) by $t_{\ge n}M$ and $t_{\le n}M$, respectively (that is, $ \tau_{\ge n }(H_M)\cong H_{t_{\ge n}M}$ and $ \tau_{\le n }(H_M)\cong H_{t_{\le n}M}$; see 
  Definition \ref{dvtt}(3) and Proposition  \ref{prtst}(\ref{itho})).

\item\label{iwfild} Let $H\opp$ be the 
  functor obtained from $H$ by means of inverting arrows both in $\cu$ and $\au$. Then there exist canonical isomorphisms  $\tau\opp_{\le n}(H\opp)\cong (\tau_{\ge -n}(H))\opp$ and $\tau\opp_{\ge n}(H\opp)\cong  (\tau_{\le -n}(H))\opp$; here $\tau\opp$ denotes virtual $t$-truncations with respect to $w\opp$.

\item\label{iwfilp5} 
  Virtual $t$-truncations of $H'$ give well-defined covariant functors. Moreover, if $H'$ is homological then all these virtual $t$-truncations are  homological and there exists a homological analogue of the long exact sequence (\ref{evtt}). 

Furthemore, if   there exists a $t$-structure $t$ on a 
 triangulated category $\cu'$ that is  anti-orthogonal to $w$ 
  (in a triangulated category $\du$ containing $\cu$ and $\cu'$), $\au=\ab$, and the functor $H'\cong H^N$ 
  for some object $N$ of $\cu'$,  then  $ \tau_{\ge n }(H')\cong H^{t_{\ge n}M}$ and $ \tau_{\le n }(H')\cong H^{t_{\le n}M}$.  
\end{enumerate}
\end{pr}
\begin{proof} Assertions \ref{iwfilp1}--\ref{iwfilp3}  are essentially given by Theorem 2.3.1 of \cite{bger}; yet pay attention to Remark \ref{rstws}(3).  Assertion \ref{iwfilp4} is given by Proposition 2.5.4(1) of ibid. (see Remark  \ref{rldual} 
  below for some detail). 

 Assertion \ref{iwfild} can be obtained by means of applying  Proposition \ref{pbw}(\ref{idual}) quite straightforwardly.
Firstly, 
reversing arrows converts any $m$-weight decomposition of $M\in \obj \cu$ into its $-m-1$-weight decomposition in $\cu \opp$ (with respect to $w\opp$). Moreover, if we reverse arrows in 
 (\ref{ecompl}) then we obtain a similar diagram in $\cu\opp$; if $m<n$ in the initial diagram then we obtain the 
 unique morphism of triangles which corresponds to $g$ and connects (the dual) $-n-1$-weight decomposition of $M'$ with the corresponding $-m-1$-weight decomposition of $M$.
Thus if we reverse 
 in the definitions $\tau_{\ge -n}(H)$ and  $\tau_{\le -n}(H)$ (for any $n\in \z$) then we obtain
the morphisms that compute $\tau\opp_{\le n}(H\opp)$ and $\tau\opp_{\ge n}(H\opp)$, respectively.

Assertion \ref{iwfilp5}  is easily seen to be dual to assertions  \ref{iwfilp1}--\ref{iwfilp3}; see the preceding assertion. 
\end{proof}

Some more 
 useful properties of virtual $t$-truncations follow from our definitions easily.

\begin{pr}\label{pwfilt}
Let $n\in \z$; assume that $\cu$ is endowed with a weight structure $w$  and $H$ is a 
 contravariant functor from $\cu$ into $\au$.

I.1. For any $i\in \z$ we have $\tau_{\le n+i}(H)\cong  \tau_{\le n}(H\circ [i])\circ [-i]$ and $\tau_{\ge n+i}(H)\cong  \tau_{\ge n}(H\circ [i])\circ [-i]$, and also $\tau_{\le n+i}(H')\cong  \tau_{\le n}(H'\circ [i])\circ [-i]$ and $\tau_{\ge n+i}(H')\cong  \tau_{\ge n}(H'\circ [i])\circ [-i]$. 

2. Assume that $F:\au\to \au'$ is an exact functor. Then $\tau_{\le n}(F\circ H)\cong F\circ \tau_{\le n}(H)$ and $\tau_{\ge n}(F\circ H)\cong F\circ \tau_{\ge n}(H)$.

II. Assume that 
  $\cuz$ is a triangulated category endowed with  a weight structure $\wz$; let $F:\cuz\to \cu$ be a weight-exact functor.

1. Then $\tau^0_{\le n}(H\circ F)\cong  (\tau_{\le n}(H))\circ F$ and $\tau^0_{\ge n}(H\circ F)\cong  (\tau_{\le n}(H))\circ F$; here $\tau^0_{\le n}(-)$ and $\tau^0_{\ge n}(-)$ denote the corresponding virtual $t$-truncations with respect to $\wz$.

2. Assume that $F$ is a full embedding, that is, $w$ restricts to a weight structure $w^0$ on $\cuz$. 

Then $\tau^0_{\le n}(H\Big|_{\cuz})\cong  \tau_{\le n}(H)\Big|_{\cuz}$ and $\tau^0_{\ge n}(H\Big|_{\cuz})\cong  \tau_{\ge n}(H)\Big|_{\cuz}$. 

3. In addition to the assumptions of the previous assertion, suppose that $\au^0$ is a (strictly full) abelian subcategory of 
  $\au$ (see \S\ref{snotata}). 
 
Then if the restricted 
 functor $H\Big|_{\cuz}$ takes it values in $\au^0$ then the same is true for 
  the restrictions to $\cuz$ of all virtual $t$-truncations of $H$.

\end{pr}
\begin{proof}
I.1. Obvious.

II. Very easy as well; recall that exact functors respect images of morphisms.

II.1. Very easy; note that for any $M\in \obj \cuz$ one can compute all 
 $(\tau_{\le n}(H)) 
   (F(M))$ and $(\tau_{\ge n}(H)) 
   (F(M))$ by means of applying $H$ to the $w$-truncations of $F(M)$ that are computed as images with respect to $F$ with respect to  $\wz$-truncations of $M$. 

2. This is just a particular case of the previous assertion; note that $H\Big|_{\cuz}=H\circ F$.

3. According to the previous assertion, it suffices to 
look at the values of (all the functors) $\tau^0_{\le n}(H\Big|_{\cuz})$ and $\tau^0_{\ge n}(H\Big|_{\cuz})$. For the latter purpose it suffices to apply assertion I.2 to the functors  $H\Big|_{\cuz}$ and the exact embedding functor $\au^0\to \au$.
\end{proof}

The following nice observation will be applied in a succeeding paper.

\begin{coro}\label{ctosham} 
Assume that there exists a $t$-structure $t$ on $\cu'$  that is orthogonal 
   to $w$ (for certain $\cu'$ and $\du$ as in Definition \ref{dort}),
   and let $F:\cu\to \eu$ be an exact functor that is weight-exact  with respect a weight structure $w^{\eu}$ on $\eu$. Suppose that  for some $M\in \obj \cu'$ the functor $H_M^{\cu}$
   is isomorphic to $H^{\eu}\circ F$, where $H^{\eu}$ is a contravariant 
    functor $\eu\to \ab$.
   
   Then for any     $n\in \z$ the functors   $H^1=H_{t_{\le n}M}$ and $H^2= H_{t_{\ge n}M}$  are isomorphic to  $(\tau^{\eu}_{\le n}(H^{\eu})) \circ
   F$ and  $(\tau^{\eu}_{\ge n}(H^{\eu})) \circ
   F$, respectively; here $\tau^{\eu}_{\le n}$ and  $\tau^{\eu}_{\ge n}$ denote the corresponding virtual $t$-truncations with respect to $w^{\eu}$.  
   
 Consequently, 
both of  $H^{1,2}$ annihilate any object of $\cu$ that is annihilated by 
$F$. 
 \end{coro}  
 
 \begin{proof}
 According to Proposition \ref{pwfil}(\ref{iwfilp4}), 
  the functors  $H^{1,2}$  are isomorphic to $ \tau_{\le n }(H_M)$ and $\tau_{\ge n}(H_M)$, respectively. According to  Proposition \ref{pwfilt}(II.1) it follows that 
  they are isomorphic to $(\tau^{\eu}_{\le n}(H^{\eu})) \circ
   F$ and  $(\tau^{\eu}_{\ge n}(H^{\eu})) \circ F$ indeed.

  Consequently,  if $F(N)=0$ for some $N\in \obj \cu$ then $H^1(N)=0=H^2(N)$; note that $H^{\eu}$ sends $0$ into $0$ since $H_M$ does.
 \end{proof}


An argument somewhat related to the proof above yields the following funny statements. Clearly some arrows can be reversed in them (to make $H$ contravariant), and 
similarly to Corollary \ref{ctosham} 
 they can be applied to (co)representable functors and their $t$-truncations with respect to (anti)orthogonal $t$-structures. 

\begin{pr}\label{portmor} 

 Let $F:\cu\to \bu$ and $H:\cu\to\to \eu$  be  additive (covariant) functors 
 and suppose that 
  $H$ annihilates all $\cu$-morphisms that are  annihilated by $F$. Choose $n\in \z$; set  $H^1= \tau_{\le n }(H)$ and $H^2=\tau_{\ge n}(H)$.
	
	1. Assume that $\bu$ is triangulated and $F$ is a weight-exact functor with respect to some weight structure $w^{\bu}$ on $\bu$.
 
 Then 
both $H^1$ and $H^2$ annihilate all $\cu$-morphisms that are  annihilated by $F$ as well.

2. Assume that $\bu$ is abelian. 

Then $H^1$  (resp. $H^2$) annihilates all $\cu$-morphisms that are  annihilated by  $\tau_{\le n }(F)$ (resp.  by $\tau_{\le n }(F)$).

 \end{pr}  
 
 \begin{proof}
  Clearly, it suffices to verify the statements in question for $H^1$; see Proposition \ref{pwfil}(\ref{iwfild}). 
 
 Next, choose 
 some $g\in \cu(M,M')$. 
Recall that (by Proposition \ref{pbw}(\ref{icompl})) there exists a commutative diagram
 \begin{equation}\label{evttmor}
 \begin{CD} w_{\le n} M@>{c}>>w_{\le n+1} M
@>{}>> M\\
@VV{g_n}V@VV{g_{n+1}}V@ VV{g}V \\
w_{\le n} M'@>{}>>w_{\le n+1} M'
@>{}>> M' \end{CD}\end{equation}
 and Definition \ref{dvtt}(1) easily yields that $\imm(H^1(g))\cong \imm (H(g_{n+1}\circ c))$.
 Thus $H^1(g)=0$ if and only if $H(g_{n+1}\circ c)=0$.

 1. 
Hence our assumption on $H$  implies that it suffices to verify the following: if $F(g)=0$
 then $F(g_{n+1}\circ c)=0$. 
 
 Now, applying $F$ to (\ref{evttmor})  we easily obtain a commutative square  $$\begin{CD} F(w_{\le n}) M@>{}>> F(M)\\
@VV{F(g_{n+1}\circ c)}V@ VV{F(g)}V \\
F(w_{\le n+1} M')
@>{}>> F(M') \end{CD}$$
Applying the uniqueness statement in Proposition \ref{pbw}(\ref{icompl}) to the weight structure $w^{\bu}$   on $\bu$ (along with the weight-exactness of $F$) we obtain that 
 $F(g_{n+1}\circ c)=0$ since $F(g)=0$. This concludes the proof.

2. 
As we have just noted (for $H$ instead of $F$) we have  $\tau_{\le n }(F)(g)=0$ if and only if  $F(g_{n+1}\circ c)=0$. If this is the case then $H(g_{n+1}\circ c)$ is zero as well (by our assumptions) and we conclude that  $H^1(g)=0$.
  \end{proof}

 \subsection{On weight range}\label{swr}

Now we define weight range and relate it to virtual $t$-truncations; some of these statements will be applied elsewhere (only). 

\begin{defi}\label{drange}
 Let $m,n\in \z$; let  $H$ be 
 a cohomological functor from $\cu$ into $\au$.

Then we will say that $H$ is {\it of weight range} $\ge m$ (resp. $\le n$, resp.  $[m,n]$)  if it annihilates $\cu_{w\le m-1}$ (resp. $\cu_{w\ge n+1}$, resp. both of these classes). 

Moreover,  functors of weight range $[0,0]$ will also said to be {\it pure}; cf. Remark \ref{rpure}(2) below.
\end{defi}

\begin{pr}\label{pwrange}
Once again, assume that $\cu$ is endowed with a weight structure $w$, $m, n\in \z$, and $H$ is a cohomological functor from $\cu$ into $\au$. 


\begin{enumerate}

\item\label{iwrvt} 
 Then  the functor $\tau_{\le n}(H)$ is of weight range $\le n$ and $\tau_{\ge m}(H)$ is of weight range $\ge m$.

\item\label{iwridemp} We have $\tau_{\le n}(H)\cong H$ (resp. $\tau_{\ge m}(H)\cong H$) if and only if $H$ is of weight range $\le n$ (resp. of weight range $\ge m$).

\item\label{iwrcomm} We have $\tau_{\le n}(\tau_{\ge m})(H)\cong \tau_{\ge m}(\tau_{\le n})(H)$.

\item\label{iwfil4} If $H$ is 
of weight range $\ge m$  then  $\tau_{\le n}(H)$ is 
  of weight range $[m,n]$.

Dually, if $H$ is 
of weight range $\le n$  then  $\tau_{\ge m}(H)$ is 
  of weight range $[m,n]$.

\item\label{iwrpure} The (not necessarily locally small) category of 
 pure cohomological functors from $\cu$ into $\au$  is equivalent to $\adfu(\hw\opp,\au)$  in the obvious natural way. 


\item\label{iwfil5} If $H$ is of weight range $\ge m$ (resp. $\le m$) then  the morphism  $H(w_{\ge m}M)\to H(M)$ is epimorphic (resp. the morphism  $H(M)\to H(w_{\le m}M)$ is monomorphic); here we take arbitrary choices of the corresponding weight decompositions of $M$ and apply $H$ to the connecting morphisms.

\item\label{iwrvan} 
Assume that $m>n$. Then the only functors of weight range  $[m,n]$ are zero ones; thus if $H$ is of weight range $\le n$ (resp. $\ge m$) then $\tau_{\ge m}(H)=0$ (resp. $\tau_{\le n}(H)=0$).

\item\label{iwfil3} The (representable) functor  $H_M=\cu(-,M):\cu\opp\to \ab$  is of weight range $\ge m$ if and only if $M\in \cu_{w\ge m}$.

\item\label{iwfilsubc}
Assume that $H$ is pure and its restriction to $\hw$ takes its values in an abelian subcategory $\auz$ of $\au$. Then the values of $H$ belong to $\au$ as well. 

\item\label{iwrort} Assume that $t$ is a $t$-structure on $\cu'$ (for some $\cu'\subset \du$, where $\du$ contains $\cu$) that is 
 orthogonal to $w$. Then for 
 $N\in \cu'_{t\le 0}$ (resp.  $N\in \cu'_{t\ge 0}$, resp. $N\in \cu'_{t= 0}$) the corresponding $\du$-Yoneda functor $H_N:\cu\opp\to \ab$ (see Definition \ref{dvtt}(3))  is of weight range $\le 0$ (resp. $\ge 0$, resp. $[0,0]$).

\end{enumerate}
\end{pr}
\begin{proof}


\ref{iwrvt}. Let $M\in \cu_{w\ge n+1}$. Then we can take $w_{\le n}(M)=0$. Thus $\tau_{\le n}(H)(M)=0$, and we obtain the first part of the assertion. It second part is easily seen to be dual to the first part. 


\ref{iwridemp}. 
This  is precisely Theorem 2.3.1(III.2,3) of \cite{bger} (up to change of notation); assertion \ref{iwrcomm} is given by part II.3 of that theorem.


\ref{iwfil4}. 
Let  $H$ be of weight range $\ge m$.  Then $\tau_{\le n}(H)\cong \tau_{\le n}(\tau_{\ge m})(H)\cong \tau_{\ge m}(\tau_{\le n})(H)$ (according to the two previous assertions). It remains to apply assertion \ref{iwrvt} to obtain the first statement in the assertion,  
whereas its second part is easily seen to be the dual of the first part (and certainly  can be proved similarly).

\ref{iwrpure}. Immediate from Theorem 2.1.2(2) of \cite{bwcp}.  

\ref{iwrvan}. 
For any $l\in \z$ and any cohomological $H$ any choice of an $l$-weight decomposition triangle (cf. (\ref{ewdm})) for $M$ gives a long exact sequence
\begin{equation}\label{eles}
\begin{gathered}
\dots \to H((w_{\le l}M)[1])\to  H(w_{\ge l+1}M)\to H(M)\\
\to H(w_{\le l}M)\to H((w_{\ge l+1}M)[-1])\to\dots
\end{gathered}
\end{equation}
The exactness of this sequence in $H(M)$ for $l=n$ immediately gives   the first part of the assertion. 
 Next, the second part is straightforward from the first one combined with assertion \ref{iwfil4}.

Assertion \ref{iwfil3} is immediate from  Proposition \ref{pbw}(\ref{iort}). 

Assertion \ref{iwfil5} is a straightforward consequence of assertion \ref{iwfil3}; just apply (\ref{eles}) for $l=m$ and for $l=m-1$, respectively. 

Assertion \ref{iwfilsubc} easily  follows  assertion \ref{iwrpure}; see also Lemma 2.1.4 of \cite{bwcp}.

\ref{iwrort}. For $N\in \cu'_{t\le 0}$ and $N\in \cu'_{t\ge 0}$ 
the weight range estimates for the functor $H_N$ prescribed by the assertion 
 are given by the definition of orthogonality, and to obtain the claim for $N\in \cu'_{t= 0}$ one should combine the first two weight range statements.

\end{proof}

\begin{rema}\label{rpure}
1. Roughly, the statements above say that virtual $t$-truncations of functors behave as if they corresponded to $t$-truncations of objects in a certain triangulated 'category of functors' (whence the name; certainly, another justification of this idea is provided by the existence of orthogonal $t$-structures statements that will be proved below).

2.  In particular, one can 'slice' any functor of weight range $[m,n]$ for $m\le n$ into 'pieces' of weight $[i,i]$ for $m\le i\le n$. Next, composing a 'slice' of weight range $[i,i]$ with $[i]$ one obtains a pure functor. 

 Recall also  
  that pure functors were studied in detail in (\S2.1 of) \cite{bwcp}; they were called pure ones due to the relation to Deligne's purity (cf. Remark 2.1.3(3--4) of ibid.).
\end{rema}

Let us also prove some more complicated statements of this sort.

\begin{pr}\label{pwrangen}

Adopt the notation and conventions of Proposition \ref{pwrange}, and assume that $H'$ and $H''$ are cohomological functors from $\cu$ into $\au$ (as well).

I. Then the following statement are fulfilled.

1. If 
 $H'$ is of weight range $\ge n$ (resp. of weight range $\le n-1$) then any transformation $T:H'\to H$ (resp. $H\to H'$) factors through the transformation $a_n$ 
 (resp. $b_{n-1}$) in the sequence  
 (\ref{evtt}).


2. If 
  $H'$ (resp. $H''$) 
    is of weight range $\ge n$ (resp. of weight range $\le n-1$)
then there are no non-zero transformations from $H'$ into $H''$.

II. Assume in addition that we have an 
exact sequence of cohomological functors from $\cu$ into $\au$ 
\begin{equation} \label{evttpc} 
H'\stackrel{a}{\to} H\stackrel{b}{\to} H'' 
\end{equation}
(that is, applying it to any  $M\in \obj \cu$ we obtain a half-exact $\au$-sequence) and for any $M\in \obj \cu$ there exists an isomorphism $\ke a(M[-1])\to \cok b(M)$.

1. If
  $H'$ is of weight range $\ge n$,    $H''$ is of weight range $\le n-1$, then there exists a canonical isomorphism of chains of morphisms of functors
  
   (\ref{evttpc})$\cong  \tau_{\ge n }(H)\stackrel{a_n}{\to} H  \stackrel{b_{n-1}}{\to} \tau_{\le n -1}(H)$;
   
    here the latter chain is a part of the sequence (\ref{evtt}).

2. $a$  factorizes as  $H'\stackrel{i'}{\to}\tau_{\ge n }(H)\stackrel{a_n}{\to} H $, where $i'$ is an isomorphism of functors,  if and only if  $b$ 
 factorizes as  $ H  \stackrel{b_{n-1}}{\to} \tau_{\le n -1}(H) \stackrel{i''}{\to} H'' $, where $i''$ is an isomorphism.
 
 III. Assume that $\cu,\cu'\subset \du$ and $H=H_{M}^{\cu}$
    for some object $M$ of $\cu'$. Then the transformation $a_n$ corresponding to $H$ can be presented as $H^{\cu}_f$ (that is, the the restriction to $\cu$ of   $\du(-,f)$; see Definition \ref{dvtt}(3)) 
 for some $\cu'$-morphism $f:L\to M$ 
  if and only if $b_{n-1}$ 
   is 
    of the form $H^{\cu}_g$ for some $\cu'$-morphism  $g:M\to R$.
   
  Moreover, the morphisms $L\stackrel{f}{\to} M\stackrel{g}{\to} R$ can be completed to a distinguished triangle in $\cu'$.
\end{pr}
\begin{proof}

 I.1. 
  The two statements in the assertion 
 are easily seen to be dual to each other; hence it suffices to consider the case where $H'$ is of weight range $\ge n$. Next, the obvious functoriality of the definition of virtual $t$-truncations gives the
following commutative square of transformations:

\begin{equation}\label{euni}
\begin{CD} 
\tau_{\ge n}H'@>{\tau_{\ge n}T}>> \tau_{\ge n}H \\
@VV{i'}V@VV{i}V \\
H'@>{T}>>H
\end{CD}\end{equation}
(cf. (\ref{evtt})).

Applying assertion \ref{iwridemp} we obtain that the transformation $i'$ is an equivalence. Hence the transformation 
$\tau_{\ge n}T$ yields the factorization in question. 

2. According to assertion 
I.1, any transformation in question factors through $\tau_{\ge n }(H'')$; thus it is zero according to Proposition \ref{pwrange}(\ref{iwrvan}). 

II.1. 
This is just a re-formulation of \cite[Theorem 2.3.1(III.4)]{bger}.

2.  Assume that $a=a_n\circ i',$  where $i'$ is an isomorphism of functors.
According to the previous assertion, to prove that $b=i''\circ b_{n-1}$ (where $i''$ is an isomorphism) it suffices to verify that $H''$ is of weight range $\le n-1$.

 Combining our assumptions on (\ref{evttpc}) with Proposition \ref{pwfil}(\ref{iwfilp3}) we obtain that for any $M\in \obj \cu$ both $H''(M)$ and $\tau_{\le n -1}(H)$ can be presented as certain extensions of the object $\ke a(M[-1])\cong \ke i'(M[-1])$ by $\cok a(M)\cong \cok i'(M)$. Consequently, $H''(M)=0$ if and only if $\tau_{\le n -1}(H)$. Thus $H''$ is of weight range $\le n-1$ since $\tau_{\le n -1}(H)$ (see Proposition \ref{pwrange}(\ref{iwrvt})).
 
 Similarly, if $b=i''\circ b_{n-1}$ then one can easily verify that $H'(M)=0$ if and only if $\tau_{\ge n }(H)(M)=0$. Consequently, $H'$ is of weight range $\ge n$ since $\tau_{\ge n }(H)$ is (see Proposition \ref{pwrange}(\ref{iwrvt})), and applying assertion II.1 
  we conclude that $a=a_n\circ i'$.

III. Assume that 
 $f\in \cu'(L,M)$. We complete $f$ to a distinguished triangle \begin{equation}\label{etrir} L\stackrel{f}{\to} M\stackrel{g}{\to} R\to L[1]. \end{equation}  
 Then the restriction of the transformation sequence  $\du(-,L\stackrel{f}{\to} M\stackrel{g}{\to} R)$  to $\cu$ clearly satisfies our assumptions on (\ref{evttpc}) in assertion II. Then assertion II.2 
   implies that $b_{n-1}$ is the corresponding restriction of $\du(-,g)$. We also obtain the "moreover" statement in our assertion.

Conversely, if $b_{n-1}\cong H^{\cu}_g$ 
 then we also can complete $g$ to a distinguished triangle of the form (\ref{etrir}). We apply assertion II.2 
 once again to obtain that $a_n\cong H^{\cu}_f$. 
\end{proof}


We also 
 prove a simple statement that was used in \cite{bsnew}.

\begin{pr}\label{pcrivtt}
For $M\in \obj \cu$ the following conditions are equivalent.

(i) $M\in \cu_{w\ge 0}$. 

(ii) $H(M)=0$ for any cohomological functor $H$ from $\cu$ into (an abelian category) $\au$ that is of weight range $\le -1$.

(iii) $(\tau_{\le-  1}H_N)(M) =\ns$ for any $N\in \obj \cu$.

(iv) $(\tau_{\le -1}H_M)(M) =\ns$.

\end{pr}
\begin{proof}
Condition (i) implies condition  (ii) by definition; clearly, (iii)   $\implies$ (iv). Next, condition (ii) implies condition (iii) according to Proposition  \ref{pwrange}(\ref{iwrvt}). 

Lastly, if $(\tau_{\le -1}H_M)(M) =\ns$ then the long exact sequence (\ref{evtt}) yields that $(\tau_{\ge -0}H_M)(M)$ surjects onto $\cu(M,M)$. Hence the morphism $\id_M$ factors through $w_{\ge  0}M$; thus $M$ belongs to $ \cu_{w\ge 0}$.
\end{proof}

\subsection{On $R$-linear categories} \label{srlin}

In some of the statements below we will need a few easy properties of $R$-linear categories. 

Throughout this paper  $R$ will be an associative commutative 
 unital ring. 

We will use the following '$R$-linearization construction'.

\begin{defi}\label{drlin}

Let $\bu$ be an $R$-linear 
category, $F:\bu\opp\to \ab$ is an  additive functor (in particular, $F$ may be $\bu'$-representable for some additive category $\bu'\supset \bu$; cf. Proposition \ref{lrlin}).

Then for $M\in \obj \bu$ 
we define the 
 multiplication by $r$ endomorphism on the group $F(M)$ 
 as  $F(r\id_{B})$. 
\end{defi}

\begin{lem}\label{lrlin} 

Let $\bu\subset \bu'$ be  additive categories, and assume that $\bu$ is $R$-linear. 

1.  Let $F:\bu\opp\to \ab$ be an  additive functor

Then for any $M\in \obj \bu$ the definition above endows 
 $F(M)$ with the structure of an $R$-module.

 2. 
 The category $\adfur(\bu\opp,R-\modd)$ of $R$-linear functors  
 is equivalent to $ \adfu(\bu\opp,\ab)$.

This equivalence is obtained by means of composing a functor $G\in  \adfur(\bu\opp,R-\modd)$ with the forgetful functor $F_R:R-\modd\to \ab$, and to obtain the converse correspondence we define the 
 $R$-module structure on $F(B)$ using assertion 1 (for any $B\in \obj\bu$).

 3.  A functor $G:\bu\opp\to R-\modd$ is representable by an object of $\bu'$ (that is, can be obtained from $H_{\bu}B'$ by means of the construction mentioned in assertion 1; here we use notation similar to Definition \ref{dvtt}(3), and $B'\in \obj \bu'$) whenever 
 $F_R\circ G$ 
 is representable by an object of $\bu'$ as a functor into abelian groups. 

 4. Assume that $\bu$ is also a triangulated category, $w$ is a weight structure on it,  $n\in \z$. 
  
  Then for an $R$-linear cohomological functor $H_R$ from $\bu$ into $R-\modd$  the functors $\tau_{\le n}(F_R\circ H)$ and $\tau_{\ge n}(F_R\circ H)$  are isomorphic to  $F_R\circ \tau_{\le n}( H)$ and $F_R \circ \tau_{\ge n}(H)$, respectively. 
 
 Consequently, the functor $\tau_{\le n}( H)$ is representable in $\bu$ if and only if $\tau_{\le n}(F_R\circ H)$ (resp. $\tau_{\ge n}(F_R\circ H)$) is (cf. assertion 2).
 \end{lem}
 \begin{proof}  
Assertion 1 is obvious.
 
 Assertions 2 
 is quite easy; 
  this is Lemma  3.11(1) 
   of \cite{bdec}. Assertion 3 follows from it immediately.
 
 Next, the first part of assertion 4 is a particular case of Proposition \ref{pwfilt}(I.2).  Hence it remains to apply assertion 3 to conclude the proof. 
 
\end{proof}

\subsection{
 Strictly orthogonal  $t$-structures (and their restrictions) 
  } 
\label{sorts}

\begin{pr}\label{pstrict}
Assume that $\cu', \cu\subset\du$ and  $w$ is a weight structure on $\cu$. Take 
\begin{equation}\label{ec12} C_1= \cu_{w\le -1}^{\perp_{\du}} \cap \obj\cu';\    C_2=\cu_{w\ge 0}^{\perp_{\du}} \cap \obj\cu'\end{equation}

1. Suppose that $C_1\perp C_2$ and a $t$-structure $t$ on $\cu'$ is orthogonal to $w$. 

Then $t$ is also  strictly  orthogonal to $w$; thus it equals $(C_2[1],C_1)$.

2. Suppose that a $t$-structure $t$ on $\cu'$ is strictly orthogonal to $w$.  Then $C_1\perp C_2$. 

Consequently, $t$ is the only $t$-structure on $\cu'$ that is orthogonal to $w$.
\end{pr}
\begin{proof}
1. We should verify that $t=(C_2[1],C_1)$. The definition of orthogonality says that  $\cu'_{t\le 0}\subset C_2[1]$ and  $\cu'_{t\ge 0}\subset C_1$. 
On the other hand, recall that $\cu'_{t\ge 0}={}^{\perp_{\cu'}} \cu'_{t\le -1}$ and $\cu'_{t\le 0}=(\cu'_{t\ge 1})^{\perp_{\cu'}} $ (see Proposition  \ref{prtst}(\ref{itperp})). Since $C_1\perp C_2$, we obtain that the converse inclusions are valid as well.
 
 2. Since $(C_2[1],C_1)$ is a $t$-structure, $C_1\perp C_2$; see the orthogonality axiom (iii) in Definition \ref{dtstr}. Next, if  a $t$-structure $t'$ on $\cu'$ is orthogonal to $w$ as well, then assertion 1 implies that $t'=t$.
\end{proof}

\begin{rema}\label{rhomortp}
 Let us formulate a simple general statement that we essentially used in the proof of Proposition \ref{pstrict}(1); we will apply it below.

Assume that $C_1,C_2,C'_1,C'_2\subset \obj \cu'$ (for some triangulated category $\cu'$), $C_1\perp C_2$, $C'_1\subset C_1$, $C'_2\subset C_2$, $C_2'=C'_1\perpp$, and $C'_1=\perpp C_2'$. Then we clearly have $C'_1\supset C_1$ and $C'_2\supset C_2$;  thus $C'_1= C_1$ and $C'_2= C_2$.


\end{rema}





Now we prove our first abstract criterion on the existence of an orthogonal $t$-structure.

\begin{theo}\label{trefl}
Assume that $\cu'\subset \cu$,  
 and  $w$ is a weight structure on $\cu$. 

 I. Then the following conditions are equivalent.

(i). There exists a $t$-structure $t$ on $\cu'$  orthogonal to $w$ (in $\cu$). 
 
(ii). There exists a $t$-structure $t$ on $\cu'$ strictly  orthogonal to $w$. 

(ii').  There exists a $t$-structure $t$ on $\cu'$  strictly orthogonal to $w$, and this is the only  $t$-structure $t$ on $\cu'$  orthogonal to $w$.

(iii). The functor  $\tau_{\ge 0}H^{\cu}_{M}=\tau_{\ge 0}H_{M}$ is $\cu$-representable  
by an object of $\cu'$ for any object $M$ of $\cu'$ (that is, $\tau_{\ge 0}H_{M}\cong H_{LM}$ for some object $LM$ of $\cu'$). 

(iv). If   $M$ is an object of $\cu'$ then  the functor  $\tau_{\le -1}H_{M}$ is isomorphic to $ H_{RM}$ 
 for some $RM\in \obj\cu'$. 

(v).  For any object $M$ of $\cu'$ and $i\in \z$ the functors  $\tau_{\ge i}H_{M}$ and $\tau_{\le i}H_{M}$ are $\cu$-representable by objects of $\cu'$.

II. If the conditions of assertion I are fulfilled, then  the corresponding Yoneda-type functor $\hrt\to \adfu(\hw\opp,\ab),\ M\mapsto H_{M}^{\cu}$ (see Definition \ref{dvtt}(3)) is fully faithful. 

Consequently, if the category $\hw$ is $R$-linear then $\hrt$ embeds into $\adfur(\hw\opp,R-\modd)$ via the equivalence provided by Lemma \ref{lrlin}(2).
\end{theo}
\begin{proof}
  Condition (ii) clearly implies condition (i).  Conditions (ii) and (ii') are equivalent according to Proposition \ref{pstrict}.
  Next, Proposition \ref{pbw}(\ref{iort}) implies that the class $C_2$  (see (\ref{ec12}) for this notation) lies in $\cu_{w\le -1}$; hence $C_1\perp C_2$ by the definition of $C_2$. Hence  Proposition \ref{pstrict} implies that $(i)\implies (ii)$.
  
  Let us verify that 
   condition (iv) implies (ii). So,  we should check that the couple $(C_2[1],C_1)$ 
  is a $t$-structure.  Now, these classes are clearly closed with respect to $\cu'$-isomorphisms, and the "shift" axiom (ii) of Definition \ref{dtstr} is obviously fulfilled as well. 
 Moreover, we have just proved that $C_1\perp C_2$; this gives the orthogonality axiom (iii).
Hence it remains to check the existence of $t$-decompositions (this is axiom (iv) of $t$-structures).
For an object $M$ of $\cu'$ we assume $\tau_{\le -1}H_{M}\cong H_{RM}$ for some $RM\in \obj \cu'$. 
By the Yoneda lemma, the 
  transformation $b_{-1}$ in (\ref{evtt}) equals 
   $\cu(-,g)$ for some $g\in \cu'(M,RM)$. 
  Then Proposition \ref{pwrangen}(III) 
  gives a $\cu'$-distinguished triangle $$LM\to M\to RM\to LM[1],$$ where  $LM\in \obj \cu'$ and $H_{LM}\cong \tau_{\ge 0}H$. It remains to note that $LM\in C_1$ and $RM\in C_2$; see Proposition \ref{pwrange}(\ref{iwrvt}). 
  
	 Next, (i) implies condition (v) according to Proposition \ref{pwfil}(\ref{iwfilp4}). Moreover, 
	 condition (v) clearly implies conditions (iii) and (iv).
	
	 Lastly, assume that our condition (iii) is fulfilled; we fix  
		$M\in \obj \cu'$ 
		 and assume that $\tau_{\ge 0}H\cong H_{LM}=\cu({-,LM})$ for some  $LM\in \obj \cu'$.  
		  We argue similarly to the proof of  ((iv)$\implies$ (ii)) above.
	By the Yoneda lemma, 
	 the corresponding transformation $a_0: \tau_{\ge 0}H\to H$ 
	can be presented
	as $\cu(-,f)$ for some $f\in \cu'(LM,M)$. 
	Hence Proposition \ref{pwrangen}(III) yields that $\tau_{\le -1}H_{M}\cong H_{RM}$ for the corresponding $RM$. Thus  condition (iii) implies (iv), and this concludes the proof.

 II. If $M\in \cu'_{t=0}$ then the definition of orthogonality implies that the  functor $H_M:\cu\opp\to \ab$ (see Definition \ref{dvtt}(3)) is of weight range $[0,0]$. Thus it suffices to apply Proposition \ref{pwrange}(\ref{iwrpure}) along with the Yoneda lemma and  Lemma \ref{lrlin}(2).
\end{proof}

Now we will describe certain restrictions of $t$ that is strictly orthogonal to $w$. 
 We start from some definitions. 

\begin{defi}\label{damain}
  Assume that $\cuz,\cu\subset \du$, and $\cuz$ is $R$-linear.
 

1.  We will say that $\auz\subset \au$ is an abelian subcategory (of an  abelian category  $\au$) if $\auz$ is a (strictly full) subcategory of $\au$ that contains the $\au$-kernel and the $\au$-cokernel of any $\auz$-morphism. 

Moreover, we will say that $\auz$ is a {\it weak Serre} subcategory of $\au$ if $\auz$ is an
 abelian subcategory that is closed with respect to $\au$-extensions.

2. Let $\au$ be a weak Serre subcategory of $R-\modd$. Then we will write $\cu_{\au}$ for
the full subcategory  of $\cu$ 
that consists of those $N\in \obj \cu$ such that for any $M\in \obj \cuz$ the $R$-module $\du(M,N)$ (see Lemma \ref{lrlin}(4)) belongs to $\au$.

3. In the case where $R$ is Noetherian, $\be$ is an infinite cardinal, and $\au$ is the category of $R$-modules with less than $\be$ generators (see Proposition \ref{pcgrlin}(3) below) we will modify this notation and write $\cu_\be$ for the corresponding $\cu_{\au}\subset \cu$.
\end{defi}

\begin{pr}\label{pcgrlin}
Assume that $\cuz, \cu\subset\du$,  
 $\cuz$ is $R$-linear, a weight structure $\wz$ on $\cuz$ is strictly orthogonal to a $t$-structure $t$ on $\cu$, and $\au$ and $\au'$ are weak Serre subcategories of $R-\modd$. 

   1. Then $\cu_{\au}$ is triangulated and $t$ restricts to it. The heart $\hrt_{\au}$ of this restriction $t_{\au}$ consists of those objects of $\hrt$ such that the corresponding functors $\hwz{}\opp\to R-\modd$  (see Proposition \ref{pwrange}(\ref{iwrpure},\ref{iwrort}) and Lemma \ref{lrlin}(2))
take values in $\au$. 
 
 2. $\cu_{\au\cap \au'}=\cu_{\au\cap \au'}\cap \cu_{\au\cap \au'}$. 
 

 3. If $R$ is Noetherian then the category of $R$-modules with less than $\be$ generators is a weak Serre subcategory of $R-\modd$.
   \end{pr}
 \begin{proof}

1. Since for any $M\in \obj \cu$ the functor $\du(M,-)$ is homological, $\cu_{\au}$ is triangulated indeed. 

 Next, to verify that $t$ restricts to $\cu_{\au}$ it suffices to that for any $M\in \obj \cu_{\au}$ the object $L_tM=t_{\ge 0}M$ belongs to  $ \cu_{\au}$ as well; see Lemma \ref{lrest}. The latter statement easily follows from Proposition \ref{pwfilt}(II.3) along with Lemma \ref{lrlin}(4). 

Lastly, any object of $\hrt_{\au}$  yields a functor $\hwz{}\opp\to\au$ just by the definition of $\cu_{\au}$. Conversely, if an object of $\hrt$ gives a functor  $\hwz{}\opp\to\au$ then the corresponding functor $\cuz\opp\to R-\modd$  takes values in $\au$ as well; see  Proposition \ref{pwrange}(\ref{iwfilsubc}). 

2.  Obvious.

3. A 
well-known fact.
\end{proof}

We will describe some geometric examples for this proposition later. 

\subsection{Coproductive extensions of weight structures}\label{scopra}



In the following definition we do not assume that $\cu$ and $\au$ are closed with respect to (small) coproducts.

\begin{defi}\label{dcoprode}
1. Let $w^0$ be a weight structure on $\cuz\subset \cu$. We will say that (a weight structure) $w=(\cu_{w\le 0}, \cu_{w\ge 0})$ is the {\it coproductive extension} of $w^0$ to $\cu$ if $\cu_{w\le 0}$ (resp.  $\cu_{w\ge 0}$) equals the smallest retraction-closed and extension-closed class of objects of $\cu$ that is closed with respect to $\cu$-coproducts (that exist in $\cu$) 
 and contains $\cu^0_{w^0\ge 0}$ (resp.  $\cu^0_{w^0\ge 0}$). 

2. It will be convenient for us to use the following somewhat clumsy  terminology: 
a cohomological functor  $H'$ from $\cu$ into $\au$ will be called a {\it cp} functor if it converts all 
$\cu$-coproducts (that is, all  those small coproduct diagrams that exist in $\cu$) into $\au$-products.
\end{defi}



\begin{rema}\label{rcoprode}
1. If $w$ is the coproductive extension of $w^0$ to $\cu$ and $w'$ is an arbitrary extension of $w^0$ to $\cu$ then Proposition \ref{pbw}(\ref{iort}) is easily seen to imply that $\cu_{w\le 0}\subset  \cu_{w'\le 0}$; cf. Proposition \ref{pbw}(\ref{icoprod}). 
Moreover, if $w$ is an extension of $w^0$ to $\cu$ 
such that $\cu_{w\le 0}\subset  \cu_{w'\le 0}$ for any other extension $w'$ then this assumption determines it canonically (since  $\cu_{w\ge 0}\supset  \cu_{w'\ge 0}$). Yet the existence of $w$ of this sort does not imply that $w$ is the coproductive extension of $w^0$ to $\cu$ since  the coproductive extension of $w^0$ to $\cu$ does not have to  exist. 
Indeed, 
the coproductive extension of $w^0$ does not exist in the case $\cuz=\ns$   and $\cu\neq \ns$; yet $w=(0,\cu)$ clearly satisfies the aforementioned minimality property. 

2. The author is not sure that 
 Definition \ref{dcoprode}(1) is 
 really clever.
  However, one can construct plenty of examples for it (see Theorem \ref{tcompws} 
  below), 
 and it also fits well with the following simple statements.
\end{rema}

\begin{pr}\label{pcopr}
Assume that $\cuz\subset \cu$, $w^0$ is a weight structure on $\cuz$, $w$ is its coproductive extension to $\cu$. 

1. Let $H$ be 
 cp (cohomological)  functor from $\cu$ into $\au$. 

 Then $H$ is of $w$-weight range $\ge 0$ (resp. $\le 0$) if and only if its restriction to $\cuz$ is of  $w^0$-weight range $\ge 0$ (resp. $\le 0$).

2. Assume that 
 $\cu,\cu'\subset \du$.

Then a $t$-structure $t$ on $\cu'$ is (strictly) orthogonal to $w$ if and only if it is (strictly) orthogonal to $w^0$.

\end{pr}
\begin{proof}
1. The "only if" part of the assertion is obvious; it only requires $w$ to be an extension of $w^0$.

The converse implication immediately from the descriptions of $\cu_{w\le -1}$ and $\cu_{w\ge 1}$ provided by Definition \ref{dcoprode}.

2. It clearly suffices to verify that  the classes $C_1=\cu_{w\le -1}^{\perp_{\du}} \cap \obj\cu' $  and $C_2=\cu_{w\ge 0}^{\perp_{\du}} \cap \obj\cu'$ (cf. (\ref{ec12}))  coincide with \begin{equation}\label{ec120} C_1^0=\cu_{w^0\le -1}^{\perp_{\du}} \cap \obj\cu' \text{  and }C^0_2=\cu^0_{w^0\ge 0}{}^{\perp_{\du}} \cap \obj\cu'  \end{equation} respectively. The latter easy statement can be obtained by means of applying assertion 1 to functors from $\cu$ that are represented by objects of $\cu'$.
\end{proof}

\begin{coro}\label{ccopr}
Assume that $\cuz,\cu'\subset \cu$, $w^0$ is a weight structure on $\cuz$ and $w$ is its coproductive extension to $\cu$. 

Then the equivalent conditions of Theorem \ref{trefl}(I) are also equivalent to the existence of a  $t$-structure $t$ on $\cu'$ that is  orthogonal to $w^0$. Moreover, this $t$-structure equals $(C_2^0[1],C_1^0)$ in the notation of (\ref{ec120}), that is, $t$ is strictly orthogonal to $w^0$. 
\end{coro}
\begin{proof} This is an obvious combination of Proposition \ref{pcopr}(2) with (the equivalence of conditions (i) and (ii) in) Theorem \ref{trefl}(I).
\end{proof}

\begin{rema}\label{rccopr}
Corollary \ref{ccopr} is a certain substitute for Theorem 2.2.5 of \cite{bvtr}; 
see Remark \ref{rbvtt} below. 
Though the author suspects that the results of the current text are not sufficient to prove loc. cit. itself, 
 they can be successfully applied to treat all the examples of loc. cit. known to the author; 
  cf. Proposition \ref{pgeomap1} below. 
\end{rema}

\section{On $t$-structures orthogonal to (co)smashing weight structures}
\label{smashb} 


In this section we study the existence of adjacent weight and $t$-structures in triangulated categories closed with respect to (co)products (these are called smashing and cosmashing ones). 

In \S\ref{smash} we 
 treat smashing weight structures (on smashing triangulated categories); these are the ones that 'respect coproducts'.

In \S\ref{smashort}  we recall the notion of 
 Brown representability for smashing triangulated categories, and 
prove Theorem \ref{tadjti}, i.e., that a weight structure $w$ on a category satisfying this condition is left adjacent to a $t$-structure if and only if $w$ is smashing.

In \S\ref{sort}  we study extensions of weight structures from subcategories of compact objects and the corresponding 
 orthogonal $t$-structures. This context will be actual for us in the succeeding section.

In \S\ref{scosm} we formulate the dual to 
 Theorem \ref{tadjti} and give some applications for it. 
 If $w$ is both cosmashing and smashing then the left adjacent $t$-structure $t$ restricts to the subcategory of compact objects of $\cu$ as well as to all other 'levels of smallness' for objects. Combining this statement with an existence of weight structures theorem from \cite{kellerw} we obtain a statement on $t$-structures extending yet another result of ibid. 

\subsection{On smashing weight structures and related notions}\label{smash} 

We  need some definitions related to (co)products. 

\begin{defi}\label{dsmash}
\begin{enumerate}
\item\label{ismcat}
 We will say that a triangulated category $\cu$  is {\it (co)smashing} if it is closed with respect to (small) 
 coproducts (resp., products).

\item\label{ismw} We will say that a weight structure $w$ on $\cu$ is (co)smashing if $\cu$ is (co)smashing and the class $\cu_{w\ge 0}$ is closed with respect to $\cu$-coproducts (resp., $\cu_{w\le 0}$ is closed with respect to $\cu$-products; cf. Proposition \ref{pbw}(\ref{icoprod})). 

\item\label{ismt} We will say that a $t$-structure $t$ on $\cu$ is (co)smashing if $\cu$ is (co)smashing and the class $\cu_{t\le 0}$ is closed with respect to $\cu$-coproducts (resp., $\cu_{t\ge 0}$ is closed with respect to $\cu$-products; cf.  Proposition  \ref{prtst}(\ref{itperp})).

\item\label{idal} 
For an infinite cardinal $\al$ 
 and a smashing $\cu$ a homological functor $H':\cu\to \au$ 
 is said to be {\it $\al$-small} if for any family $N_i\in \obj \cu$, $i\in I$, we have
$H'(\coprod_{i\in I} N_i)=\inli_{J\subset I,\ \#J<\al}H'(\coprod_{j\in J}N_j)$ (i.e., the obvious morphisms $H'(\coprod_{j\in J}N_j)\to H'(\coprod N_i)$ form a colimit diagram; note that this colimit is filtered). 
\end{enumerate}
\end{defi}

Let us now prove some properties of these notions and relate them to virtual $t$-truncations. 

\begin{pr}\label{psmash}
Assume that $w$ is a smashing weight structure on $\cu$, $H':\cu\to \au$ is a homological functor (where $\au$ is an abelian category), 
$n\in \z$,  and $\al$ is an infinite cardinal. 
  Then the following statements are valid.

\begin{enumerate}
\item\label{ismash1}
If $\al'\ge \al$ 
  then any $\al$-small functor is also $\al'$-small.

\item\label{ismash2} $H'$ is $\alz$-small if and only if it respects coproducts.

\item\label{icoprh} The class $\cu_{w=0}$ is closed with respect to $\cu$-coproducts.

\item\label{icopr2} 
Coproducts of $w$-decompositions are weight decompositions as well.

\item\label{icopr6p} Assume that $\au$ is an AB4* 
  category and 
a cohomological functor $H$ from $\cu$ into $ \au$ is a cp one. Then   $\tau_{\ge  n}(H)$  and $\tau_{\le n}(H)$  are cp functors as well.

\item\label{icopr6n} Assume that $\au$ is an AB5 
  category and $H'$ 
   is an $\al$-small functor.  
	 Then the functors $\tau_{\ge  n} (H')$  and $\tau_{\le  n}(H')$  are $\al$-small as well.
\end{enumerate} 
\end{pr}
\begin{proof}
\ref{ismash1}. Assume that 
 $H'$ is an $\al$-small functor; fix  an index set $I$ and certain $N_i\in \obj \cu$. Then for any $J\subset I$ we have 
$H'(\coprod_{j\in J} N_j)=\inli_{J'\subset J,\ \#J'<\al}H'(\coprod_{j'\in J'}N_{j'})$. Combining these facts for all $J\subset I$ 
  one easily obtains that $H'(\coprod_{i\in I} N_i)=\inli_{J\subset I,\ \#J<\al'}H'(\coprod_{j\in J}N_j)$. 

\ref{ismash2}. 
Since $H'$ is additive,   $H'(\coprod N_i)=\inli_{J\subset I,\ \#J<\alz}H'(\coprod_{j\in J}N_j)$ if and only if $H'$  respects coproducts (since this colimit will not change if one will consider only those $J$ that consist of a single element only). 

\ref{icoprh}. This is an easy consequence of  Proposition \ref{pbw}(\ref{iort}); see Proposition 2.3.2(1) of \cite{bwcp}. 

\ref{icopr2}. This is an easy consequence of  Proposition \ref{pbw}(\ref{icoprod}) (combined with the well-known Remark 1.2.2 of \cite{neebook}); it is given by Proposition 2.3.2(3) of \cite{bwcp}.  

\ref{icopr6p}. 
According to Proposition \ref{pwfilt}(I.1), it suffices to verify that the functors  $\tau_{\ge  2}(H)$  and $\tau_{\le 0}(H)$ are cp ones for any cp functor $H$. 
For 
a family $\{M_i\}$ of objects of $\cu$ we choose certain 
 $0$ and $1$-weight decompositions for all of the $M_i$ (see Remark \ref{rstws}(2)) and all $j\in \z$.
According to the previous assertion,  
for $M=\coprod M_i$ we can take $w_{\le j}M=\coprod w_{\le j} M_i$ and $w_{\ge j}M=\coprod w_{\ge j} M_i$ for $j=0,1,2$.
Moreover, one clearly can  describe the unique morphisms $w_{\le 0}(\coprod M_i)\to w_{\le 1}(\coprod M_i)$  and $w_{\ge 1}(\coprod M_i)\to w_{\ge 2}(\coprod M_i)$ compatible with these decomposition triangles (see Proposition \ref{pbw}(\ref{icompl}) and Definition \ref{dvtt}(1))  as the coproducts of the corresponding morphisms for $M_i$.  Applying our assumptions on $H$ and $\au$ we obtain that  $\tau_{\ge  2}(H)(\coprod M_i)\cong \prod \tau_{\ge  2}(H)(M_i)$ and 
 $\tau_{\le  0}(H)(\coprod M_i)\cong \prod \tau_{\le 0}(H)(M_i)$.

Similarly, to prove assertion \ref{icopr6n} it suffices  
 to verify that  the functors $\tau_{\ge  2}(H')$  and $\tau_{\le 0}(H')$ are $\al$-small whenever $H'$ is. 
One takes the same weight decompositions along with their coproducts corresponding to all subsets $J$ of $I$ of cardinality less than $\al$. Since the colimits in question are filtered ones,  the AB5 assumption on $\au$ allows to compute 
$\inli_{J\subset I,\ \#J<\al}\imm(H'(\coprod_{j\in J}w_{\le 0}N_j)\to H'(\coprod_{j\in J}w_{\le 1}N_j))$ and 
$\inli_{J\subset I,\ \#J<\al}\imm(H'(\coprod_{j\in J}w_{\ge 1}N_j)\to H'(\coprod_{j\in J}w_{\ge 2}N_j))$
as the corresponding images of colimits to obtain the statement in question.
\end{proof}


\subsection{On the existence of $t$-structures adjacent to smashing weight structures}\label{smashort} 

To formulate the main results of this section and discuss examples to them we  need some more definitions.

\begin{defi}\label{dcomp}
Let  $\cu$ be a smashing triangulated category, $\cp$ is a subclass of $ \obj \cu$, $\cu'$ is an arbitrary triangulated category. 

\begin{enumerate}
\item\label{idbrown} We will say that  $\cu$  satisfies the  {\it Brown representability} property whenever any cp functor (see Definition \ref{dcoprode}(2)) from $\cu$ into $\ab$ is representable. 

Dually, we will say that 
 $\cu'$ satisfies the  {\it dual Brown representability} property if $\cu'$ is cosmashing  and any  functor  from $\cu'$ into $\ab$ that respects products is corepresentable (i.e., if $\cu'{}\opp$ satisfies the   Brown representability assumption).

\item\label{dsmall}
For an infinite cardinal $\al$ an object $M$ of $\cu$ is said to be {\it $\al$-small} (in $\cu$) if the functor $H^M=\cu(M,-):\cu\to \ab$ 
 is $\al$-small (see Definition \ref{dsmash}(\ref{idal})). We will write $\cu^{(\al)}$ for the  (full)  subcategory $\cu^{(\al)}$ of $\cu$ that consists of $\al$-small objects.

Moreover, $\alz$-small objects of $\cu$ (those correspond to  functors that respect coproducts) will also said to be {\it compact}.

\item\label{dlocal}
We will say  that a 
 triangulated subcategory $\cuz$ of $ \cu$ is  {\it localizing}  whenever it is closed with respect to $\cu$-coproducts. 
Respectively, we will call the smallest localizing  subcategory of $\cu$ that contains a given class $\cp\subset \obj \cu$  the {\it  localizing subcategory of $\cu$ generated by $\cp$}. We will say that both $\cp$ and the corresponding full subcategory of $\cu$ 
 {\it generate} $\cu$.

\item\label{dcgen}
We will say that a generating class of objects $\cp$ (as well as the subcategory $C$ of $\cu$ with $\obj C=\cp$) {\it compactly generates} $\cu$ if $\cp$ is also 
  essentially small and consists of compact objects (of $\cu$). 



\item\label{dgenw} We will say that a class $\cp'\subset \obj \cu'$ {\it generates} a weight structure $w$  (resp.  a $t$-structure $t$) on  $\cu'$ whenever $\cu'_{w\ge 0}=(\cup_{i>0}\cp'[-i])\perpp$ (resp. $\cu'_{t\le 0}=(\cup_{i>0}\cp'[i])\perpp$).

\end{enumerate}
\end{defi}

\begin{rema}\label{rcomp}
1. Recall that $\cu$ satisfies both the Brown representability property and its dual whenever it is  compactly generated; 
 see Proposition 8.4.1, Proposition 8.4.2, Theorem 8.6.1, and  Remark 6.4.5  of \cite{neebook}.  Moreover, the Brown representability property is fulfilled whenever $\cu$ is just {\it $\alo$-perfectly generated} (see Definition 8.1.4 and Theorem 8.3.3 of ibid.). 



2.  A class $\cp'$ as above is easily seen to determine weight and $t$-structures it generates on $\cu'$ (if any) completely; see either of Proposition 2.4(1) (along with \S3) of  \cite{bvt} or  Proposition  \ref{prtst}(\ref{itperp}) and Proposition \ref{pbw}(\ref{iort}) above.

\end{rema}

Now we prove our first 'practical' existence of $t$-structures result; see Definitions \ref{dsmash} and \ref{dcomp} for the 
 notions mentioned in our theorem.

\begin{theo}\label{tsmash}
Let $w$ be a weight structure on $\cu$, where $\cu$ (is smashing and) satisfies the Brown representability property.

Then there exists a $t$-structure $t^l$ left adjacent to $w$ if  and only if $w$ is smashing. Moreover, $t^l$ is cosmashing (if exists) 
 and its heart is equivalent to the category of those  additive functors $\hw\opp\to \ab$ that respect products.

\end{theo}
\begin{proof}
 The "only if" assertion is essentially given by Proposition 2.4(6) of \cite{bvt} (cf. \S3 of ibid.; 
the statement is also very easy for itself).

Conversely,  assume that $w$ is smashing. According to Theorem \ref{trefl}(I) the existence of $t^l$ is equivalent to the representability of $\tau_{\le 0}H_{M}$ for any representable functor $H_{M}$. Next, Proposition \ref{psmash}(\ref{icopr6p}) says that   $\tau_{\le 0}H_{M}$ is a cp functor since $H_{M}$ is. Hence $\tau_{\le 0}H_{M}$ is representable by the Brown representability assumption, and we obtain  that $t^l$ exists indeed. 

Next,  the category $\cu$  is cosmashing according to  Proposition 8.4.6 of  \cite{neebook} (since it satisfies the Brown representability property).  Moreover, $t^l=(C_2[1],C_1)$, where $C_1= \cu_{w\le -1}^{\perp}$  and $C_2= \cu_{w\ge 0}^{\perp}$ according to Theorem \ref{trefl}(I) (alternatively, one can apply Proposition \ref{portadj}). 
 Hence the class $\cu_{t^l\ge 0}$  is closed with respect to $\cu$-products; thus $t^l$ is cosmashing as well. 

Lastly, since $t^l=(C_2[1],C_1)$, the class $\cu_{t^l=0}$ equals $(\cu_{w\ge 1}\cup \cu_{w\le -1})^{\perp}$. Hence Proposition \ref{pwrange}(\ref{iwrpure},\ref{iwrort}) 
along with Proposition 2.3.2(8) of \cite{bwcp} 
 gives the calculation in question.
\end{proof}

\begin{rema}\label{rsmashex}
 Now let us discuss examples to Theorem \ref{tsmash}. 

1. According to Theorem 5 of \cite{paucomp},  any set $\cp$ of compact objects of $\cu$ generates a (unique) smashing weight structure (see Definition \ref{dcomp}(\ref{dgenw}) and Remark \ref{rcomp}(2)). Thus one may say that there are lots of smashing weight structures on $\cu$ whenever there are 'plenty' of compact objects in it (see Theorem 4.15 of \cite{postov} for a certain justification of this claim for derived categories of commutative rings). 
  Thus our theorem yields a rich collection of $t$-structures; yet 
  cf. Remark \ref{rneetgen} below. 
  
  We will say more on smashing weight structures in Remark \ref{rexam} below. Moreover, we will also discuss the application of Theorem \ref{tsmash} to semi-orthogonal decompositions; 
   this gives a generalization of \cite[Corollary 2.4]{nisao}.
  
  2. The author has applied  Theorem \ref{tsmash}  in 
  \S4.2 of \cite{bok}  to obtain and study an interesting family of $t$-structures on various motivic triangulated categories. 

\end{rema} 

\subsection{Weight structures extended from subcategories of compact objects, and orthogonal $t$-structures}\label{sort}

To construct certain weight structures we need the following statements that appear to be rather well-known.

\begin{pr}\label{pstar}
 Assume that $A$ and $B$ 
 are extension-closed classes of objects of $\cu$. 

 I. Assume that $A\perp B[1]$. Then the class $A\star B$  (of all extensions of elements of $B$ by elements of $A$) is extension-closed as well.

II. Assume in addition that $\cu$ is smashing, and $A$ and $B$ are closed with respect to $\cu$-coproducts. 

1. Then  $A\star B$  is closed with respect to $\cu$-coproducts as well.

2. Assume that $A$ 
 is closed either with respect to $[-1]$ or with respect to $[1]$. Then $A$ 
 is retraction-closed in $\cu$.  
\end{pr}
\begin{proof}
All of these statements are rather easy.

Assertions I and II.1 
 immediately follow from Proposition 2.1.1 of \cite{bsnew}. Assertion II.2 is a straightforward consequence of Corollary 2.1.3(2) of 
 ibid. (see  Remark 2.1.4(4) of ibid.)
\end{proof}

Now we prove that weight structures extend from subcategories of compact objects to the localizing subcategories they generate. 
The (proof of the) first part of the following theorem is quite similar to the corresponding arguments in \S2 of \cite{bsnew}.


\begin{theo}\label{tcompws}
Assume that $\cu$ is a smashing triangulated category, a subcategory $\cuz $ of $ \cu^{(\alz)}$ generates $\cu$, 
 and $\wz$ is a weight structure on $\cuz$.

I.1. Then $\wz$ extends uniquely to a smashing weight structure $w$ on $\cu$.  

2. $w$ is 
 the coproductive extension of $\wz$ to $\cu$.

3.  $\cu_{w= 0}$ consists of all retract of all (small) $\cu$-coproducts of elements of $\cu^0_{\wz= 0}$.

4. Assume that  $\al$ is a {\it regular} cardinal, that is, $\al$  cannot be presented as a sum of less then $\al$ cardinals that are less than $\al$. 
Take $\cu_{\al}$ to be the smallest triangulated subcategory of $\cu$ that contains $\cuz$ and is closed with respect to coproducts of less  than $\al$ objects. Then there exists a 
coproductive extension $w_{\al}$ of $\wz$ to $\cu_{\al}$. Moreover, $\cu_{\al}{}_{w_{\al}= 0}$  consists of all retracts of all coproducts of less than $\al$ elements of $\cu^0_{\wz= 0}$.

II. Assume in addition that the category $\cuz$ is essentially small. 

1. Then there exists a $t$-structure $t$ (left) adjacent to $w$. 

2. 
 $t$ is strictly  orthogonal to $\wz$; hence $t$ is both smashing and cosmashing.

3. 
$\hrt$ is equivalent (in the obvious way) to the category $\adfu(\hwz{}\opp,\ab)$.

\end{theo}
\begin{proof}
I.1,2. Proposition \ref{pbw}(\ref{iuni}) clearly yields the following: if a coproductive extension of $\wz$ to $\cu$ exists then it equals the only smashing extension of $\wz$ to $\cu$.

Set 
$E_1$ (resp. $E_2$) for the smallest extension-closed subclass of $\obj \cu$ that is closed with respect to coproducts and contains $\cu^0_{\wz\le 0}$ (resp.  $\cu^0_{\wz\ge 0}$). To prove our assertions, it suffices to verify that $(E_1,E_2)$ is a weight structure on $\cu$. 

Firstly, 
 axiom (ii) of Definition \ref{dwstr} (for $\wz$) easily implies that  $E_1\subset E_1[1]$ and $E_2[1]\subset E_2$.  
Combining this statement with  
Proposition \ref{pstar}(II.2)  we obtain that $E_1$ and $E_2$ are retraction-closed in $\cu$. 

Next, the compactness of the elements of $\cu^0_{\wz\le 0}$ in $\cu$ implies that the class $\cu^0_{\wz\le 0}\perpp$ is closed with respect to coproducts. Since it is also extension-closed and contains $\cu^0_{\wz\ge 1}$ by the axiom (iii) of Definition \ref{dwstr}, this orthogonal 
 contains $E_2[1]$, i.e., $\cu_{w\le 0}\perp E_2[1]$. Thus $\cu_{w\le 0}\subset \perpp E_2[1]$, and since the latter class is closed with respect to coproducts and extensions, we obtain that $E_1\perp E_2[1]$ (cf. the proof of \cite[Lemma 1.1.1(2)]{bokum}).

Let us now prove the existence of weight decompositions, i.e., 
 for $E=E_1\star E_2[1]$ we should prove $E=\obj \cu$. Now, axiom (iv) in Definition \ref{dwstr} (applied to $w^0$) implies that $E$  contains $\obj \cuz$, and 
 Proposition \ref{pstar}(I, II.1) implies that $E$ is also extension-closed and closed with respect to $\cu$-coproducts. Hence $E=\obj \cu$, and we obtain that $w=(E_1,E_2)$ is a weight structure on $\cu$ indeed. Clearly, this weight structure is smashing.


3. Denote  our candidate for $\cu_{w=0}$ by $C$. 
Firstly we note that  $C\subset \cu_{w=0}$ since the latter class is  closed with respect to coproducts according to Proposition \ref{psmash}(\ref{icoprh}). 

Applying Proposition \ref{pbw}(\ref{isplit})  we obtain that $C$  is  extension-closed (since this class is clearly additive); obviously, it is also closed with respect to coproducts. 

Next we apply 
  Proposition \ref{pstar}(I, II.1) once again to obtain that   
  the class $C\star \cu_{w \ge 1}$ is extension-closed and closed with respect to coproducts; hence this class coincides with $\cu_{w \ge 0}$. Thus for any $M\in \cu_{w= 0}$ there exists its weight decomposition $LM\to M\to RM\to LM[1]$ with $LM\in C$.  Since $M$ is a retract of $LM$ according to Proposition \ref{pbw}(\ref{iwdmod}),  we obtain that $M\in C$.
  
  4. The proof is similar to that of the preceding assertions; cf. also the proof of \cite[Theorem 2.2.1]{bsnew}.

II.1. The category $\cu$ is compactly generated by $\cuz$ in this case; hence $\cu$ satisfies the Brown representability property (see Remark \ref{rcomp}(1)). Next, $w$ is smashing; thus a $t$-structure $t$ adjacent to it exists according to Theorem \ref{tsmash}. 

2. $t$ is cosmashing according to Theorem \ref{tsmash}. 
It is 
strictly  orthogonal to $\wz$ by Corollary  \ref{ccopr}; thus 
 $\cu_{t\le 0}=\cu^0_{\wz\ge 1}\perpp$. 
 Since for any object $N$ of $\cuz$ the class 
  $N\perpp$ is closed with respect to coproducts, 
  $t$ is smashing. 

3. According to Theorem \ref{tsmash}, the category $\hrt$ is equivalent to the category of those functors from $\hw\opp$ into $\ab$ that respect products. Thus it remains to apply the description of $\hw$ provided by assertion I.3.
\end{proof}

\begin{rema}\label{rcuzkar}

1. The restriction of our theorem to the case where $\wz$ is bounded 
 was essentially established in \S4.5 of \cite{bws}; cf.  Theorem 3.2.2 of \cite{bwcp} and Remark 2.3.2(2) of \cite{bsnew} for some more detail.

2. Note that one can easily construct plenty of examples  for our theorem such that  $\wz$ is unbounded. 

Indeed, 
there exist lots of unbounded weight structures on (essentially) small triangulated categories; 
 in particular, one can start from a (say, bounded) weight structure on some small non-zero $\cuz\subset  \cu^{(\alz)}\subset \cu$, choose a regular $\al>\alz$, and apply Theorem \ref{tcompws}(I.4) on the corresponding category $\cu_{\al}^0\subset \cu$ (that is the smallest subcategory of $\cu$ that is closed with respect to coproducts of less than $\al$ objects and contains $\cuz$).
 Next, 
if $\cuz=h\mathcal{C}_0$ is the homotopy category  of a 
 small stable $\infty$-category $\mathcal{C}_0$ then 
  it compactly generates the category $\cu=h\operatorname{Ind}\mathcal{C}_0$; see Remark 1.4.4.3 of \cite{lurha} (cf. also the proof of Corollary 1.4.4.2 of loc. cit.).  
 
 Consequently, if $\cuz=h\mathcal{C}_0$ is a non-zero small category endowed with a weight structure $\wz$  then one can 
  extend it to an unbounded weight structure on some triangulated category $\cu'_0\supset \cuz$, and $\cu'_0\subset  \cu'^{(\alz)}\subset \cu'$ for some compactly generated category $\cu'$.
  \end{rema}

Now let us combine Proposition \ref{pcgrlin} with Theorem \ref{tcompws}(II).

\begin{pr}\label{pcgrlintba}
Adopt the notation and assumptions of Theorem \ref{tcompws}(II);
assume in addition that $\cuz$ is $R$-linear and $\au$ is a weak Serre subcategory of $R-\modd$ (see Definition \ref{damain}). 

 1.  Then $\cu_{\au}$ is triangulated and $t$ restricts to it.
 
 2. The heart of this restriction is naturally equivalent to the category of 
  $R$-linear functors 
   $\hwz{}\opp\to \au$. 
 
 3. Let $\be$ be an infinite cardinal; 
 assume that $R$ is Noetherian and the $R$-module $\cu(M,N)$ has less than $\be$ generators for any $M,N\in \obj \cuz$. Then $\cuz\subset \cu_\be$  (cf. Proposition \ref{pcgrlin}(3)). 
   \end{pr}
 \begin{proof}
1. This is particular case of Proposition \ref{pcgrlin}(1).

2. Immediate from Theorem \ref{tcompws}(II.3) combined with Proposition \ref{pwrange}(\ref{iwfilsubc}). 

3. Obvious.
\end{proof}

Now we recall a 'geometric' setting where Proposition \ref{pcgrlintba} can be applied. 

\begin{pr}\label{pgeomap1}
  Assume that $R$ is Noetherian and a scheme $X$ is projective over $\spe R$ 
 (that is, $X$ is a closed subscheme of the projectivization $Y$ of a vector bundle $\mathcal{E}$ over $\spe R$); take $\cu=D(\operatorname{Qcoh}(X))$ (the derived categories of quasicoherent sheaves on $X$) and $\cuz=\cu^{(\alz)}$.
 
 1. Then $\cu$ is compactly generated, $\cuz=D_{perf}(X)\subset\cu$ (the subcategory of perfect complexes on $X$), and $\cu_{\alz}=D_{coh}(\operatorname{Qcoh}(X))$; here we use the notation of 
 Definition \ref{damain}, whereas a complex $N\in \obj D(\operatorname{Qcoh}(X))$ 
  belongs to $D_{coh}(\operatorname{Qcoh}(X))$ whenever all  of its 
  cohomology sheaves $H^i(N)$ are coherent.\footnote{Note that $H^0$ is actually a homological functor from the 
	 derived category of (any) abelian category; so $H^i(N)=H^0(N[i])$ in contrast to the convention introduced in \S\ref{snotata}.} 
  
  2.  Take $T$ to be a   subset of $S=\spe R$ stable under specialization, that is, $T$ is 
a union of closed subsets of $S$ (see 
 \cite[Tags 0EES, 00L1]{stacksd}). 
Then the category $\au^T$ of  $R$-modules supported on $T$ is a weak Serre subcategory of 
 $R-\modd$, and the corresponding category $\cu_{\au^T}$ 
  (see Definition \ref{damain}(2)) consists of all those 
 objects of $\cu$ 
  the sections of whose cohomology sheaves (note that those are $R$-modules) 
   are supported on $T$.
   
   Consequently, the category $\cu_{\au^T\cap R-\mmodd}$  consists of  those 
 objects of $\cu$  whose cohomology sheaves are coherent and  whose sections are supported on $T$.
  \end{pr}
	\begin{proof}
	The statement that $\cu$ is compactly generated by $D_{perf}(X)$ is well-known; see Theorem 
	4.7(1) of \cite{bdec}.
The calculation of $\cu_{\alz}$ is given by part 3 of loc. cit., and the calculation of $\cu_{\au^T}$ is given by Proposition 
 4.14(2) of ibid. It remains to apply Proposition \ref{pcgrlin}(3) to conclude the proof.
\end{proof}

\subsection{$t$-structures 
 anti-orthogonal to cosmashing weight structures}\label{scosm}

Once again, we refer to Definitions \ref{dsmash} and \ref{dcomp}.

\begin{theo}\label{tcosm}
Assume that $w$ is a cosmashing weight structure on $\cu$, and  $\cu$ satisfies the dual Brown representability property; let $\al$ be an infinite cardinal. 

1. Then there exists a smashing $t$-structure $t^r$ right adjacent to $w$ and $\hrt^r$ is equivalent to the category of those additive functors $\hw\to \ab$ that respect products.

2. The category $\cu^{(\al)}\subset \cu$ is triangulated.

3. Assume that $w$ is also smashing. Then for any infinite cardinal $\al$ 
the $t$-structure $t^r$ given  by assertion 1 restricts to the subcategory $\cu^{(\al)}$ of $\cu$. 
 Moreover,  this restricted $t$-structure 
 $t^{(\al)}$  is the only $t$-structure on $\cu^{(\al)}$ that is  anti-orthogonal to $w$.  
\end{theo}
\begin{proof}

1. This is just the categorical dual to Theorem \ref{tsmash}.

2. This is the easy Lemma 4.1.4 of \cite{neebook} (if one ignores $w$). 

3. The uniqueness of a $t$-structure on $\cu^{(\al)}$ that is anti-orthogonal to $w$ is given by the dual to 
Corollary  \ref{ccopr}.
 Next, for any object $M$ of $\cu^{(\al)}$ 
the functor $H^M$ is $\al$-small by the definition of  $\cu^{(\al)}$. Now, for $M'=t^r_{\ge 0} M$ Proposition \ref{pwfil}(\ref{iwfilp5}) says that  $H^{M'}\cong \tau_{\ge 0}H^M$. Hence the   functor $H^{M'}$ is $\al$-small as well according to Proposition \ref{psmash}(\ref{icopr6n}), and we obtain that $M'$ is an object of $\cu^{(\al)}$. Thus 
  $t^r$ restricts to $\cu^{(\al)}$ indeed; see Lemma \ref{lrest}(1). 
\end{proof}

\begin{rema}\label{rcosmall}
Dualizing Theorem \ref{tcosm}(3) one obtains that  the  $t$-structure $t^r$ provided by Theorem \ref{tsmash} restricts to the levels of a certain cosmallness filtration on $\cu$. Yet it appears that all the levels of this filtration are zero in 'reasonable cases'. 

\end{rema}

Let us now verify that 
Theorem 3.1 of \cite{kellerw}  (that essentially generalizes Theorem 3.2 of  \cite{konk}) gives an example for the setting of Theorem \ref{tsmash}(II.2), and study the corresponding structures in detail.

\begin{coro}\label{ckeller}
Let $\au$ be an 
additive  subcategory of the subcategory $\cu^{(\alz)}$ of $\cu$  that compactly generates $\cu$, 
   and assume that the category $\adfu(\au,\ab)$ is semi-simple and  $\obj \au\perp_{\cu}\cup_{i<0}\obj\au[i]$.\footnote{One may say that $\au$ is {\it anti-connective}; cf. Definition  \ref{dsilt} below.} 

1. Then there exist a smashing and cosmashing weight structure $w$ and a $t$-structure $t$ on $\cu$  that are generated by $\obj \au$, and  they are right adjacent. 

2. $t$ restricts to the subcategory $\cu^{(\al)}$ (cf. Theorem \ref{tcosm}(2)) 
  for any infinite cardinal $\al$. Moreover (for $\al=\alz$) 
  the corresponding class $\cu^{(\alz)}_{t^{(\alz)}\le 0}$ (resp. $\cu^{(\alz)}_{t^{(\alz)}\ge 0}$) is the envelope of $\cup_{i\le 0}\obj \au[i]$ (resp. of $\cup_{i\ge 0}\obj \au[i]$) in $\cu$ (see \S\ref{snotata}).
\end{coro}
\begin{proof}
1.  
 By Theorem 3.1 of \cite{kellerw}, 
   there exists a weight structure $w$ on $\cu$  such that $\cu_{w\ge 0}=(\cup_{i>0}\cp[-i])\perpp$ (i.e. $w$ is generated by $\cp=\obj \au$) and  $\cu_{w\le 0}=(\cup_{i>0}\cp[i])\perpp$.
 Since the category $\cu$ is compactly generated by $\au$, it satisfies the dual Brown representability property (by the aforementioned Theorem 8.6.1 and  Remark 6.4.5  of \cite{neebook}).  Next, $w$ is obviously smashing and cosmashing. Applying Theorem \ref{tsmash}(II.1), we obtain the existence of a smashing $t$-structure $t$ that is right adjacent to $w$. Since $\cu_{w\le 0}=\cu_{t\le 0}$, we obtain that $t$ is generated by $\cp$ (as a $t$-structure) as well.  

2. $t$ restricts to the subcategory $\cu^{(\al)}$ for any infinite cardinal $\al$ according to  Theorem \ref{tcosm}(3).
Thus it remains to prove that the classes $\cu_{t\le 0}\cap\obj \cu^{(\alz)}=\cu_{w\le 0}\cap\obj \cu^{(\alz)}$ and 
  $\cu_{t\ge 0}\cap\obj \cu^{(\alz)}$ are the envelopes in question. The latter statement 
	 easily follow from  Theorem  3.2.1(2) of \cite{bgroth} 
	 (see Remark 3.2.2(2) of ibid.).\footnote{Note also that Theorem 4.5(1,2) of \cite{postov} gives this statement under the assumption that $\cu$ is a 'stable derivator' triangulated category.}  \end{proof}

\begin{rema}\label{rkeller}
\begin{enumerate}
\item\label{ikel1}
Thus we obtain a serious generalization of 
 the existence of a (certain) $t$-structure on $\cu^{(\alz)}$ statement provided by  Theorem 7.1 of \cite{kellerw}.  In particular, we do not need any differential graded algebras (and can construct stable $\infty$-exampes for Corollary \ref{ckeller}; cf. Remark \ref{rcuzkar}(2)).    
 Note however that our arguments do not give the description of $t^{\alz}$ in terms of 'generators' that would be similar to that in loc. cit. 


\item\label{ikel2} Now we try to study the question which 
 $t$-structures on $\cu^{(\alz)}$ extend to examples for our corollary.

So, assume that $\cu$ is an arbitrary triangulated category 
 and $t'$ is a $t$-structure on $\cu^{(\alz)}$, and take $\au'=\hrt'$. 
Then  we have $\obj \au' \perp(\cup_{i>0}\obj \au'[-i])$ by the orthogonality axiom for $t'$.

Thus any 
essentially small abelian subcategory $\au$  of $\hrt'$ whose objects are semi-simple satisfies all the assumptions of our corollary except the one that $\au$  compactly generates $\cu$. Hence we can apply our corollary to the  localizing subcategory $\cu'$ of $\cu$ generated by $\obj \au$. 

\item\label{ikel3} Now assume in addition that $\cu$ is compactly generated, $t'$ is bounded, and $\hrt'$ is a length category (cf. Theorem 7.1 of \cite{kellerw}). Then the category $\cu^{(\alz)}$ is essentially small according to Lemma 4.4.5 of \cite{neebook}; 
hence $\hrt'$ is essentially small as well, and we can take $\au$ to be its subcategory of semi-simple objects. 

Since $t'$ is bounded and $\hrt'$ is a length category,  
 the category $\cu^{(\alz)}$ is densely generated (see \S\ref{snotata}) by $\obj \au$; hence  $\au$ is easily seen to generate  $\cu$ (in the sense of Definition \ref{dcomp}(\ref{dlocal})). 
 Thus one can apply Corollary \ref{ckeller} to this setting.
Moreover, it is easily seen that our assumptions on $t'$ (combined with part 2 of our corollary) imply that the corresponding $t$-structure $t^{\alz}$ coincides with $t'$.

\end{enumerate}
\end{rema}


\section{On functors and $t$-structures related to essentially bounded weight structures}\label{ststreb}

The main goal of this section is to describe some more 
  orthogonal $t$-structures 
   in the case where the corresponding weight structures satisfy certain boundedness assumptions.

   To formulate our statements in a more general case we recall (in \S\ref{sebound}) the notions of essentially bounded (above, below, or both) objects and weight structures; we motivate 
	 these 
	  definitions in Remarks \ref{rkarl}(2,3), \ref{rsatur}(2), and \ref{rpgeomap2} below. 
  We also study weight structures of these types as well as  virtual $t$-truncations with respect to them and their relation with certain essential weight-boundedness for functors.
  
  In \S\ref{sebwt} we apply the aforementioned results to the construction of certain adjacent and orthogonal $t$-structures. We apply our general statements to  various derived categories of (quasi)coherent sheaves. In particular, we are able to generalize some central results of \cite{bdec}.
  

\subsection{Essentially bounded weight structures and their relation to cohomological functors}\label{sebound}

Now we will prove some bounded versions of the results of \S\ref{swr}. To formulate them in the most general case  we need some (somewhat clumsy) definitions and statements closely related to earlier papers of the author.
 Yet the reader can avoid much of this theory if she applies it to $w$-bounded below  (resp. above) objects and weight structures only; note that those are clearly {\it essentially bounded} below  (resp. above). 

\begin{defi}\label{dpkar}
Assume that $w$ is weight structure on $\cu$; $M\in \obj \cu$.

1.  We  say that $M$ is left (resp., right) {\it $w$-degenerate} (or {\it weight-degenerate} if the choice of $w$ is clear) if $M$ belongs to $ \cap_{i\in \z}\cu_{w\ge i}$ (resp.    to $\cap_{i\in \z}\cu_{w\le i}$).

2.  We 
  say that 
 $M$ is {\it essentially $w$-positive}  (resp. {\it essentially $w$-negative}) if 
  for $\tm=M\bigoplus M[1]$ (resp. $\tm=M\bigoplus M[-1]$) there exists a distinguished triangle
  \begin{equation}\label{etm} 
  \resizebox{4.2in}{\height} {$RD(\tm)\to \tm\to Y\to RD(\tm)[1] 
  {\text{ (resp. }}  X\to \tm\to LD(\tm)\to X[1])$}
  \end{equation} 
  such that $Y\in \cu_{w\ge 0}$ (resp. $X\in  \cu_{w\le 0}$) and $RD(\tm)$ (resp. $LD(\tm)$) is right (resp. left) degenerate.

 3. We 
  say that 
  $M$ is {\it essentially $w$-bounded below}  (resp. {\it above}) 
 if $M[-n]$ is essentially $w$-positive  (resp.  essentially $w$-negative) for some $n\in \z$.
 We will 
  call the shift of the triangle (\ref{etm}) corresponding to $M[-n]$ by $[n]$ (cf. (\ref{etmc}) below) a {\it $w$-degenerate decomposition} of $\tm$.
  
  4. $w$ is said to be {\it essentially bounded below}  (resp. above) if all objects of $\cu$ are essentially $w$-bounded below  (resp. above). 
  
  We say that $w$ is essentially bounded if it is essentially bounded both above and below. 
  
  5. A cohomological functor $H$ from $\cu$ will be said to be {\it locally bounded below}  (resp. above)\footnote{This terminology was inspired by 
  Remark 0.2 of \cite{neesat}.}
   if for any $M\in \obj \cu$ we have $H^i(M)=0$ for all $i\ll 0$ (resp. $i\gg 0$); recall that $H^i(M)=H(M[-i])$. 
  
  $H$ is said to be  {\it locally bounded } if it is locally bounded  both above and below. 
 
\end{defi}

Some nice properties of essential boundedness are closely related to earlier papers.

\begin{pr}\label{peboundchar}
 Assume that $w$ is weight structure on $\cu$; $M\in \obj \cu$.

1. Assume that $M$ is essentially 
$w$-bounded below  (resp. above). 
 Then the $w$-degenerate decomposition of $\tm$ 
is canonical. 
Being more precise, the distinguished triangle 
  \begin{equation}\label{etmc} 
  \resizebox{4.2in}{\height} {$RD(\tm)\to \tm\to Y\to RD(\tm)[1] 
  {\text{ (resp. }}  X\to \tm\to LD(\tm)\to X[1]),$}
  \end{equation} 
    where $RD(\tm)$ is right (resp. $LD(\tm)$ is left) $w$-degenerate
  and $Y$ (resp. $X$) is $w$-bounded above (resp. below)
  is canonically determined by $M$ (and does not depend on the choice of the corresponding $n$).

2. Take those couples $(m,n)$ such that $m\le n\in \z$ and $n<0$ (resp. $m>0$).

The following 
conditions are equivalent.
\begin{enumerate}
\item\label{iwf6} $M$ is essentially 
$w$-positive  (resp. negative). 

\item\label{iwf6n} $M$  is a retract of some $ M'\in \obj \cu$ 
 such that there exists a distinguished triangle 
 $$ \resizebox{4.6in}{\height} {$RD(M')\to M'\to Y\to RD(M')[1] 
  {\text{ (resp. }}  X\to M'\to LD(M')\to X[1]),$} $$
  where $Y\in \cu_{w\ge 0}$ (resp. $X\in  \cu_{w\le 0}$) and $RD(M')$ (resp. $LD(M')$) is right (resp. left) degenerate.

\item\label{iwf7} $H(M)=0$ 
 if $H$ is (a cohomological functor) of weight range $[m,n]$ 
  and $(m,n)$ satisfies the conditions above.

\item\label{iwf8} $H(M)=\ns$  for $H=\tau_{\le n }(H_I)$ 
 whenever $I\in \cu_{w\ge m}$ and any $(m,n)$ as above.

\item\label{iwf9} $H(M)=\ns$, where $H=\tau_{\le n }(H_{I_0})$ and $I_0$ is a  fixed choice of  $w_{\ge m}M$ and any $(m,n)$ as above. 
\end{enumerate}

\end{pr} 
\begin{proof}
1. Assume first that $M$ is essentially 
$w$-bounded below and there exist distinguished triangles
$$T_i=(RD_i(\tm)\to \tm\to Y_i\to RD_i(\tm)[1])$$ of the type (\ref{etmc})
for $i=1,2$. Assume that  $Y_i\in \cu_{w\ge n_i}$ for  $i=1,2$ and some $n_i\in \z$. We should prove that these triangles are canonically isomorphic. 

We can assume that $n_1\ge n_2$. Then both $T_1$ and $T_2$ are $m$-weight decompositions of $\tm$ for any $n<n_2$. Applying the uniqueness statement in Proposition \ref{pbw}(\ref{icompl}) one can easily obtain $T_1\cong T_2$. In particular, one can note that both $T_1$ and $T_2$ give weight decompositions of $\tm$ that {\it avoid weight $n_2-1$}; see Theorem 2.2.1(9) of \cite{bkw}. Now we apply loc. cit. to obtain $T_1\cong T_2$.
  
The case of an essentially $w$-bounded above $M$  is just the   categorical dual of 
 the essentially 
  bounded below one; see Proposition \ref{pbw}(\ref{idual}).

2. One may say that this is a one-sided unbounded analogue of \cite[Theorem 2.5(I)]{bkwt}. 

Firstly, conditions \ref{iwf6} and \ref{iwf6n} are equivalent by Corollary 3.1.4 (resp. Theorem 3.1.3) of \cite{bkw}. 
Moreover, these statements also yield that condition \ref{iwf6n} follows from \ref{iwf7}; recall here that pure functors (cf. condition 7 of Corollary 3.1.4 and condition 8 of Theorem 3.1.3 in ibid.) are the functors of weight range $[0,0]$.\footnote{Actually, in loc. cit. mentions homological pure functors $\cu\to \au$; yet one can just reverse the arrows in the target category.} The converse implication is 
  valid as well; see 
   conditions 2 and Theorem 2.5(I) of \cite{bkwt}.

Lastly, loc. cit. says that conditions  \ref{iwf7}--\ref{iwf9} 
 become  equivalent if one fixes any $m\le n\in \z$. Combining these equivalences for all  $m\le n\in \z$ such that $n<0$ (resp. $m>0$) we obtain that our versions of conditions  \ref{iwf7}--\ref{iwf9} are equivalent as well.
 \end{proof}

Let us apply this proposition to pure functors.

\begin{coro}\label{cpureb}
Assume that $\wz$ is an essentially bounded below (resp. above,  both) weight structure on $\cuz$, and $H$ is a pure functor from $\cuz$. 

Then $H$ is locally bounded below (resp. above,  both).
\end{coro}
\begin{proof}
This statement is an immediate consequence of our definitions along with Proposition \ref{peboundchar}(II); see condition \ref{iwf7} in it.
 \end{proof}

\begin{pr}\label{peboundtr}
Assume that $w$ is weight structure on $\cu$, 
$M$ is essentially 
$w$-bounded below  (resp. above), and 
  $H$ is a cohomological functor from $\cu$.

1. Assume that $M[-n]$ is essentially $w$-positive (resp. $w$-negative, for some $n\in \z$), $m<n-1$ (resp. $m> n+1$), and $H$ is a cohomological functor from $\cu$. 
Then $\tau_{\le m}(H)(\tm)\cong H(RD(\tm))$ (resp.  $\tau_{\ge m}(H)(\tm)\cong H(LD(\tm))$). 
Consequently, 
$\tau_{\le m}(H)(M)$  (resp.  $\tau_{\ge m}(H)(M)$ ) is a retract of $H(RD(\tm))$ (resp. of  $H(LD(\tm))$).

2. Assume that $w$ is essentially bounded below (resp. above) and for any right (resp. left) weight-degenerate $D\in \obj \cu$ we have  $H^i(D)=0$ for $i\ll 0 $ (resp.  for $i\gg 0 $).
Then for any $m\in \z$ 
 the functor $\tau_{\le m}(H)$ (resp. $\tau_{\ge m}(H)$) is locally bounded below (resp. above).

3. Assume that $w$ is essentially bounded below (resp. above) and $H$ is  locally bounded below (resp. above). Then all $\tau_{\le m}(H)$ and $\tau_{\ge m}(H)$ are locally bounded below (resp. above) as well.

4. Assume that $w$ is essentially bounded and $H$ is  locally bounded. Then all $\tau_{\le m}(H)$ and $\tau_{\ge m}(H)$ are locally bounded as well.

\end{pr} 
\begin{proof}
1. Since $M$ is a retract of $\tm$, it suffices to prove the first part of the assertion.

Now, assume that $M[-n]$ is essentially $w$-bounded below. 
Similarly to the argument above, note that we can take $w_{\le i}(\tilde M)=RD(\tm)$ for $i<n$. Recalling the definition of $\tau_{\le m}(H)$ (see Definition \ref{dvtt}(1)), we obtain that $\tau_{\le m}(H)(\tm)\cong \imm(\id_{H(RD(\tm))})=H(RD(\tm))$.

The 
case of an essentially $w$-bounded above $M$ can be treated similarly. Moreover, it can be easily deduced from the   essentially bounded below case; see Propositions \ref{pbw}(\ref{idual}) and \ref{pwfil}(\ref{iwfild}).


2. Once again, we give the proof in the   essentially bounded below case; the essentially bounded above case is very much similar and  
 can be deduced from it.

According to the previous assertion, it suffices to verify that $(\tau_{\le m}(H))^i(RD(\tm))=0$ for $i\ll 0 $ and any $M\in \obj \cu$. Thus it remains to recall that $RD(\tm)$ is right weight-degenerate.

3. This assertion immediately follows from the previous one if we recall the exact sequence (\ref{evtt}).

4. This assertion can be immediately obtained by means of combining the bounded above and below cases of the previous one.

\end{proof}


\begin{rema}\label
	{rkarl}
	If $w$ is {\it weight-Karoubian}, that is,  if $\hw$ is Karoubian then one modify 
	 the definition of essentially $w$-bounded above and below objects 
	 by setting $\tm=M$ in the triangles (\ref{etm}); see Theorems 2.3.4 and  3.1.3  and Corollary 3.1.4 of \cite{bkw}. Respectively, one can compute  $\tau_{\le m}(H)(M)$ for any $M\in \obj \cu$ using the corresponding canonical decomposition similarly to Proposition \ref{peboundtr}(1). 
	
	
	2. Obviously, any semi-orthogonal decomposition couple is an essentially bounded weight structure (if we permute the components); see condition \ref{isodw} of Proposition \ref{psod}.
	
		Furthermore, one may say that general essentially bounded above and below weight structures are 
	'mixes' of bounded above and below ones and semi-orthogonal decompositions. Indeed, if $w$ is weight-Karoubian and essentially bounded above or below  then the 
	statements mentioned 
	 in part 1 of this remark yield the existence of a certain semi-orthogonal decomposition. 
	
3. Respectively, the main results of this section can be used to generalize 
 some statements on   semi-orthogonal decompositions from \cite{bdec}. We will say more one this in Remark \ref{rpgeomap2}(2)  below.	
\end{rema}

\subsection{
On $t$-structures corresponding to essentially bounded weight structures}\label{sebwt}

Now we pass to the existence  of $t$-structures orthogonal to essentially bounded weight structures. 
We start with a nice 'bounded $R$-linear analogue' of Theorem \ref{tsmash}. Once again, we assume 
  that $R$ is a commutative Noetherian ring.

\begin{defi}\label{dsatur}
Assume that $\cu$ is an $R$-linear category.

We will say that $\cu$ is {\it $R$-saturated} if the  representable functors from $\cu$ into $R-\modd$  are exactly the ones that are locally bounded and take values in the subcategory $R-\mmodd$ of finitely generated $R$-modules.

\end{defi}

\begin{coro}\label{csatur} 
1. Assume that $\cu$ is an $R$-saturated category, and $w$ is an essentially bounded weight structure on it.

Then there exists a $t$ structure $t$ (left) adjacent to $w$. 
  Its heart is equivalent to the category of $R$-linear functors from $\hw\opp$ into the category of finitely generated $R$-modules.

2. Take $T$ to be a   union of closed subsets of $S=\spe R$, and   $\au^T$ to be the category of 
   $R$-modules supported on $T$ (cf. Proposition \ref{pgeomap1}(2)). 

Then the  category $\cu_{\au^T}$ corresponding to $\cuz=\cu$ (see Definition \ref{damain}(2))  
is  triangulated and $t$ restricts to it. Moreover, the heart of this restriction is equivalent to the category of $R$-linear functors from $\hw\opp$ into the category of finitely generated $R$-modules supported on $T$.

3. Assume that $X$ is  regular and proper over $\spe R$. 
 Then  the category $\cu=D^{perf}(X)$  is $R$-saturated.
Moreover, $\cu$ is equivalent to $\cu\opp$.

Consequently, if $w$ is an essentially bounded weight structure on $\cu$ then there exists 
a $t$-structure on $\cu$ that is  left  adjacent to $w$ and also a $t$-structure that is  right adjacent to it.

\end{coro}
\begin{proof}
1. According to Theorem \ref{trefl}(I), it suffices to verify that for any representable functor $H$ the functor $\tau_{\le 0}H$ is representable as well (in the $R$-linear sense; see Lemma \ref{lrlin}). Now, $\tau_{\le 0}H$ is  locally bounded according to Proposition \ref{peboundtr}(4), and takes is values in (the abelian subcategory of) 
 finitely generated $R$-modules according to Proposition \ref{pwfilt}(II.3).

Next, $\hrt$ embeds into the category of $R$-linear functors  $\hw\opp\to R-\modd$ according to Theorem \ref{trefl}(II). Now, any functor $\hw\opp\to R-\mmodd$ factors through a pure cohomological functor from $\cu$ into $R-\mmodd$; see  Proposition \ref{pwrange}(\ref{iwfilsubc}). Lastly, any functor of this sort is locally bounded by Corollary \ref{cpureb}; hence it is representable (see Definition \ref{dsatur}(1)).

2. 
 Recall that $\au^T$ is a weak Serre subcategory of $R-\modd$; 
  see Proposition \ref{pgeomap1}(2). 
 Hence it remains to apply Proposition \ref{pcgrlin}(1). 
 
 3. The first part of the assertion is a particular case of \cite[Corollary 0.5]{neesat}; note that in this case  $D^{perf}(X)$ is equivalent to the bounded derived category of $X$ (cf. Remark 1.4(3) of \cite{bdec}).
 
 Next, the self-duality of $\cu$ is a well-known consequence of Grothendieck duality; see \cite[Tags 0AU3, 0DWG, 0BFQ] {stacksd}. 
 Thus it remains to apply assertion 1 to conclude the proof (see Proposition \ref{pbw}(\ref{idual}); note that $w\opp$ is clearly essentially bounded as well).
\end{proof} 

\begin{rema}\label{rsatur}
1. In all the examples 
 of $R$-saturated categories known to the author  one can achieve the same result by means of applying 
  (the somewhat more clumsy) 
 Corollary \ref{crling}  instead; see Theorem 
 4.7(1) and Proposition 4.11 
 of \cite{bdec}.

2. 
Similarly to Remark 
\ref{rexam}(1) below one can apply Corollary \ref{csatur}(1,2) to semi-orthogonal decompositions. One obtains that if $\cu$ is saturated, $\cu'=\cu_{\au^T}$ (note that 
 $\cu'=\cu$ if $T=\spe R$), and $(\ro,\lo)$ is a semi-orthogonal decomposition of $\cu$, then the couple $((\obj\ro)^{\perp}\cap \obj \cu',  (\obj\lo)^{\perp}\cap \obj \cu'=\ro \cap \obj \cu')$ is a  a semi-orthogonal decomposition of $\cu'$. Consequently, if $\ro$ is a left admissible subcategory of $\cu$ then it is also admissible in it. This statement generalizes  Proposition 2.6 of \cite{bondkaprserre}, where the case $R=\com$ was considered.

Now, it is well-known that the categories of the type $D_{perf}(X)$ often admit non-trivial semi-orthogonal decompositions (in particular, in the case where $R$ is a field). Hence there can exist unbounded weight structures on them; see Remark \ref{rkarl}(2). On the other hand the author 
 conjectures that any weight structure on   $D_{perf}(X)$ is essentially bounded; hence one can apply Corollary \ref{csatur}(1,2) to it. 

This observation motivated the author to consider essentially bounded weight structures in this section; cf. Remark \ref{rpgeomap2} 
 below. 
 
 3. Clearly, if $\cu$ is $R$-saturated  and the category $\au$ is a weak Serre subcategory of $R-\modd$ then the corresponding category $\cu_{\au}$ depends on $\au\cap R-\mmodd$ only. Now,   $\au\cap R-\mmodd$ is clearly a weak Serre subcategory of $R-\mmodd$ 
   and  
 %
  any  weak Serre subcategory of $R-\mmodd$ of $R-\mmodd$ consists of finitely generated $R$-modules supported at some $T$ as in Corollary \ref{csatur}(2); see Theorem A of \cite{taka}.\footnote{Note that extension-closed  abelian subcategories are called {\it coherent} ones in (Definition 2.3(1) of) \cite{taka}.}
   Consequently, it does not make sense to consider any $\au$ distinct from $\au^T$ in Corollary \ref{csatur}(2).

\end{rema}

Now we pass to orthogonal structures (in distinct categories). 

\begin{theo}\label{tcoprb}
Assume that $\cuz,\cu'\subset \cu$, $w^0$ is an essentially bounded below (resp. above, both) weight structure on $\cuz$, $w$ is its coproductive extension to $\cu$, and there exists a  $t$-structure $t$ on $\cu'$ that is  orthogonal to $w^0$. 
Set $\cu^{'0}_+$ (resp.  $\cu^{'0}_-$,  $\cu^{'0}_b$) to be the subcategory of $\cu'$ 
 that consists of those $M\in \obj \cu'$ such that the functor  $H_{M}^{\cuz}$ 
is locally finite below  (resp. above, both).

Then the category  $\cu^{'0}_+$ (resp.  $\cu^{'0}_-$,  $\cu^{'0}_b$) is triangulated and $t$ restricts to it.  
 Moreover, this
 restriction of $t$ is strictly  orthogonal to $w^0$, and its heart coincides with $\hrt$. 
\end{theo}
\begin{proof} The subcategories $\cu^{'0}_+$, $\cu^{'0}_-$,  $\cu^{'0}_b$ of $\cu'$ are easily seen to be triangulated regardless of any weight structures. 
Next, the strict orthogonality assertion immediately follows from Corollary \ref{ccopr}.

It remains to prove the existence of the corresponding restricted $t$-structures and calculate the heart. 
Similarly to the proof of Proposition \ref{pcgrlin}(1), Lemma \ref{lrest}(1) reduces the existence of this restriction to the following statement:  for $M$ that belongs $ \cu^{'0}_+$ (resp. $\cu^{'0}_-$,  $\cu^{'0}_b$) 
 the  functor  $\tau_{\ge 0}H_{M}^{\cu}=\tau_{\ge 0}H_{M}$ is $\cu$-representable  by an object of the corresponding subcategory. Now,  $\tau_{\ge 0}H_{M}\cong H_L$ for some object $L$ of $\cu'$ (see 
 Proposition \ref{pwfil}(\ref{iwfilp4})). To prove that this $L$ actually belongs to $ \cu^{'0}_+$ (resp. $\cu^{'0}_-$,  $\cu^{'0}_b$)  
 we should check that the restricted functor 
 $\tau_{\ge 0}H_{M}^{\cuz}$ (see  Proposition \ref{pwfilt}(II.2))  is locally finite below  (resp. above, both). The latter statement 
 is given by Proposition \ref{peboundtr}(3--4). 

Lastly, Corollary \ref{cpureb} implies that any object of $\hrt$ belongs to $\cu^{'0}_+$ (resp.  $\cu^{'0}_-$,  $\cu^{'0}_b$) indeed.
\end{proof}

\begin{coro}\label{crling} 
Assume that  
 $\cu$ is a smashing $R$-linear   triangulated category, 
an essentially small subcategory $\cuz $ of $ \cu^{(\alz)}$ generates it, 
 and $\wz$ is an essentially bounded below (resp. above, both) weight structure on $\cuz$. Choose a  weak Serre subcategory  $\au\subset R-\modd$.
 

Then the $t$-structure $t$ on $\cu_{\au}$  provided by Proposition \ref{pcgrlin}(1) 
  restricts to the intersection subcategory $\cu_{\au}\cap  \cu^{'0}_+$ (resp. $\cu_{\au}\cap \cu^{'0}_-$,  $\cu_{\au}\cap \cu^{'0}_b$), and the heart of this restriction is naturally equivalent to the category of 
  $R$-linear functors from $\hwz{}\opp$ into 
   $\au$.
\end{coro}
\begin{proof}
According to Lemma \ref{lrest}(2), it suffices to combine Proposition \ref{pcgrlin}(1) with Theorem \ref{tcoprb}. 
\end{proof}

Now we relate the example provided by Proposition \ref{pgeomap1} to the conditions of Corollary \ref{crling}.

 \begin{pr}\label{pgeomap2}
  Assume that a scheme $X$ is projective (see Proposition \ref{pgeomap1}) over $\spe R$ (where  $R$ is Noetherian); take $\cu=D(\operatorname{Qcoh}(X))$ (the derived category of quasicoherent sheaves on $X$), 
    and $\cuz=\cu^{(\alz)}=D_{perf}(X)$ (see Proposition \ref{pgeomap1} once again).
	
	1. Then the intersection subcategories $\cu_{\alz}\cap  \cu^{'0}_+$, $\cu_{\alz}\cap \cu^{'0}_-$,  and $\cu_{\alz}\cap \cu^{'0}_b$ equal $D^+_{coh}(\operatorname{Qcoh}(X))$,   $D_{coh}^-(\operatorname{Qcoh}(X))$, and  $D_{coh}^b(\operatorname{Qcoh}(X))$, respectively; here a complex $N\in D_{coh}(\operatorname{Qcoh}(X))$ 
  belongs to  $D^+_{coh}(\operatorname{Qcoh}(X))$ (resp.  $D_{coh}^-(\operatorname{Qcoh}(X))$,   $D_{coh}^b(\operatorname{Qcoh}(X))$) whenever 
    its (coherent) cohomology sheaves $H^i(N)$ vanish  
	 for $i\ll 0$ and (resp. for $i\gg 0$, in both cases).
	 
	2.  Take $T$ to be a   union of closed subsets of $S$. Then to characterize the intersections of  corresponding category $\cu_{\au^T\cap R-\mmodd}$ 
  (see Definition \ref{damain}(2)) with  $\cu^{'0}_+$, $\cu^{'0}_-$, and $ \cu^{'0}_b$ one should add to the aforementioned conditions the assumption that 
  the cohomology sheaves of $N$ 
   are supported on $T$.

	\end{pr}
	\begin{proof}
1. The calculation of  these categories  is given by Theorem 4.7(3) of \cite{bdec}. 

2. Immediate from assertion 1 combined with Propositions  \ref{pgeomap1}(2) and Proposition \ref{pcgrlin}(2).
\end{proof}

	\begin{rema}\label{rpgeomap2}
	1. 
	Thus one can apply Corollary \ref{crling} to all the intersection subcategories mentioned in Proposition \ref{pgeomap2}.
	
	2. 
	  Corollary \ref{crling} also implies that for any essentially bounded weight structure on $D_{coh}^b(\operatorname{Qcoh}(X))\cong D^b(\operatorname{coh}(X))$ (see \cite[Tag 0FDA]{stacksd}) there  exists an anti-orthogonal $t$-structure on $D_{perf}(X)$; see Proposition 4.11 
	  of \cite{bdec}.
	
	The author conjectures that there do not exist bounded weight structures on  $D^b(\operatorname{coh}(X))$ if $X$ is singular. On the other hand, one may  start from a bounded weight structure on a component of a semi-orthogonal decomposition of $X$ (of arbitrary length; see Definition 2.9(II.2) of ibid.) to obtain an essentially bounded weight structure on the whole $\cu$; see Remark \ref{rkarl}(2) and Proposition 2.10 of ibid.
	
	3. Applying Remark \ref{rkarl}(2) one can prove that Corollary \ref{crling} allows to generalize all the statements in Theorem 4.12(I) of ibid. 
	 from semi-orthogonal decompositions to essentially bounded (above, below, or both) weight structures. 
	 It would also be very interesting to generalize part II of ibid. (somehow).  
	\end{rema}

\section{
The existence of orthogonal  weight structures}\label{sortw}

In this section we (try to) construct certain weight structures orthogonal to given $t$-structures. 

For this purpose, in \S\ref{srconstrw} 
 we (recall and) prove some statements related to the hearts of weight structures. 

In \S\ref{sconstrw} we prove that the existence of an 
  adjacent weight structure is closely related to the existence of enough projectives in the heart of $t$. 
This assumption ensures the existence of  weight structure on certain ('rather large') subcategories of $\cu$ that are orthogonal to $t$. Under certain assumptions, we also obtain a weight structure adjacent to $t$.

In \S\ref{sconstrwsat} we apply the results of the previous section to obtain 
a weight structure adjacent to a $t$-structure in the case where the category $\cu\opp$ is $R$-saturated (whenever certain assumptions on $t$ are fulfilled). 

\subsection{On 
 hearts and $\pm$-orthogonality of structures }\label{srconstrw}

 \begin{defi}\label{dsilt}
 We will say that a 
  subcategory $\hu\subset \cu$ 
 is {\it connective} (in $\cu$) if $\obj \hu\perp (\cup_{i>0}\obj (\hu[i]))$.
 
 We will say that $\hu$ is {\it silting}  if it is connective and densely generates $\cu$ (see \S\ref{snotata} and cf. \S3.1 of \cite{koenigyang}). 
\end{defi}


  We 
   recall rather well-known statements.

\begin{pr}\label{pcneg}

1. If $w$ is a weight structure on $\cu$ and  $\hu\subset \hw$ then $\hu$ is connective. 

2. If $\hu$ is a silting subcategory of $\cu$ 
 then the envelopes  (see \S\ref{snotata}) $\cu_{v\le 0}$ and $\cu_{v\ge 0}$ of the classes $\cup_{i\le 0} \obj \bu[i]$ and $\cup_{i\ge 0} \obj \bu[i]$, respectively, give a  bounded weight structure $v$ on $\cu$, and 
 $\underline{Hv}$ is the $\cu$-retraction closure of the additive subcategory generated by $\hu$ (i.e., of the subcategory whose objects are $\bigoplus H_i$ for finite sets of $H_i\in \obj \hu$). 

We will say that this 
 $v$ is {\it densely generated by $\hu$}. 

\end{pr}
\begin{proof}
Assertion 1 immediately follows from the orthogonality axiom in Definition \ref{dwstr}, and assertion 2  is contained in Corollary 2.1.2 of \cite{bonspkar}; see also Theorem 5.5
of \cite{mendoausbuch}.
 \end{proof}



Till the end of this 
section we 
assume that 
 $\cu'\subset \cu$ 
 and $t$ is a $t$-structure on $\cu$.

Let us now 'split' Definition \ref{dort}(3).

\begin{defi}\label{dpmort}

Assume that $w$ is a weight structure on $\cu'$.

1. We will  say that $w$ and $t$  are {\it $-$-orthogonal} (resp. {\it $+$-orthogonal})   if
 $\cu_{t\le -1}=  {\cu'_{w\ge 0}}^{\perp_{\cu}}$ (resp. $\cu_{t\ge 1}={\cu'_{w\le 0}}^{\perp_{\cu}}$).

2. We will  write $P_t$ for ${}^{\perp_{\cu}} (\cu_{t\le -1}\cup \cu_{t\ge 1})$; $P'_t=P_t\cap \obj \cu'$.

3. We will say that $t$ is {\it left (resp. right) non-degenerate} if $\cap_{i\in \z}\cu_{t\ge i}=\ns$ (resp. $\cap_{i\in \z}\cu_{t\le i}=\ns$). 

\end{defi}

 Let us establish some properties of these definitions.

\begin{pr}\label{phw1}
Let $t$ be a $t$-structure on $\cu$, $P\in P_t$.

1. 
The functor $\cu(P,-)$ restricts to an exact functor $E^P:\hrt\to \ab$, and we have $\cu(P,-)\cong E^P\circ H_0^t$. 

2. 
$\cu(P,-)\cong \cu(P,H_0^t(-))\cong \hrt (H_0^t(P), H_0^t(-)),$ 
and the first of these transformations 
 comes from the transformations $ \id_{\cu}\to t_{\le 0}$ and $H_0^t\to t_{\le 0}$ (see Proposition  \ref{prtst}(\ref{itcan},\ref{itho})). 
Moreover, $H_0^t(P)$ is projective in $\hrt$.

 3. Assume that $t$ is cosmashing 
 and $\cu$ satisfies the dual Brown representability property. Then $H_0^t$ gives an equivalence of (the full subcategory of $\cu$ given by) $P_t$ with the subcategory of projective objects of $\hrt$.
\end{pr}
\begin{proof}
I. If $0\to A_1\to A_2\to A_3\to 0$ is a short exact sequence in $\hrt$ then $A_1\to A_2\to A_3\to A_1[1]$   is  a distinguished triangle according to Proposition \ref{prtst}(\ref{itha}).  Applying the functor  $\cu(P,-)$ to it and recalling the definition of 
 $P_t$ we obtain an exact sequence  $\ns=\cu(P,A_3[-1])\to E^P(A_1)\to E^P(A_2)\to E^P(A_3)\to E^P(A_1[1])=\ns$; hence $E^P$ is exact indeed. Moreover, the functors $\\cu(P,-)$ and $ E^P\circ H_0^t$  are homological and annihilate both $\cu_{t\le -1}$ and $\cu_{t\ge 1}$; hence they are isomorphic.

2. The first of these isomorphisms is provided by assertion 1. The 
remaining statements are rather easy; 
 they are given by 
Lemma 1.3 of \cite{extproj}.

3. We should prove that any projective object $P_0$ of $\hrt$ 
 lifts to $P_t$. 

Now, the functor $H_0^t$ respects products according to the easy Lemma 1.4 of \cite{neesat}  (applied in the dual form; cf. also Proposition 3.4(2) of \cite{bvt}); thus the composition $G^{P_0}=\hrt(P_0,-)\circ H_0^t:\cu\to \ab$ respects products as well. Moreover, $G^{P_0}$ is obviously a homological functor. Thus it is corepresentable by some $P\in \obj \cu$ that clearly belongs to $P_t$, and it remains to apply the previous assertion.
\end{proof}

\begin{pr}\label{phw2}
Assume that   $t$  is  a $t$-structure on $\cu$  
orthogonal to a weight structure $w$ on $\cu'\subset \cu$. 


1. Then 
 $\cu'_{w\ge 0}= \cu_{t\ge 0}\cap \obj \cu$,  $\cu'_{w\le 0}= \perpp (\cu_{t\ge 1})\cap \obj \cu$, and $\cu'_{w=0}=P'_t$. 

2. If $t$ is $+$-orthogonal to $w$ then $\cu_{t\ge 0}$ is closed with respect to $\cu$-products.


3. Assume 
  that $t$  is  $-$ or $+$-orthogonal to $w$. 
 Then 
$\{E^P:\ P\in \cu_{w=0}\}$ is a conservative collection of functors $\hrt\to \ab$ (cf. Remark \ref{rcons} below).

4. Conversely, assume that functors of the type $E^P$ for $P\in \cu_{w=0}$ form a conservative collection and $t$ is right (resp. left) non-degenerate. Then $t$  is  $-$ (resp. $+$) right  orthogonal to $w$.
\end{pr}
\begin{proof}


1.  
The argument is rather similar to the proof of 
Theorem \ref{trefl}(I). We will use the notation $(C_1,C_2)$ for the couple $(\perpp\cu_{t\ge 1}\cap \obj \cu',\ \cu_{t\ge 1}\cap \obj \cu')$.

Since $w$ is orthogonal to $t$, we have $\cu'_{w\le 0}\subset  C_1$ and $\cu'_{w\ge 1}\subset  C_2$. Next, $C_1\perp C_2$, and applying Proposition \ref{pbw}(\ref{iort}) we 
obtain $w=(C_1,C_2[-1])$; see Remark \ref{rhomortp}. 
Hence $\cu'_{w=0}=C_1\cap C_2[-1]=P'_t$.

2. Obvious.

3. Since all of these functors are exact (see Proposition \ref{phw1}(1)), it suffices to verify that for any non-zero $P\in \cu_{t=0}$ there exists $M\in \cu'_{w=0}$ such that $\cu(M,P)\neq \ns$. 

Now, the functor $H^{\cu'}_P$ is of weight range $[0,0]$ by Proposition \ref{pwrange}(\ref{iwrort}). 
Consequently, if   $M\perp P$ for all $M\in \cu'_{w=0}$ then 
Proposition \ref{pwrange}(\ref{iwrpure}) 
implies $\obj \cu' \perp P$. 
 Combining this statement with either the $-$ or the $+$-orthogonality of $t$ to $w$ we immediately obtain $P=0$ (i.e., a contradiction).

4. If $t$ is right (resp. left)  non-degenerate, it suffices  to verify that $H_i^t(N)=0$ whenever $i<0$ (resp. $i>0$), $N\in \obj \cu$,   and $M\perp N$ for all $M\in \cu'_{w=i}$. For this purpose it is clearly sufficient to check that $H_0^t(N)=0$ whenever
$M\perp N$ for all $M\in \cu'_{w=0}$. The latter statement  immediately follows from our assumption on $\cu'_{w=0}$ along with assertion 1.
\end{proof}

\begin{rema}\label{rcons} 
 Since the functors of the type $E^P$ that we consider in parts 
 3--4 of our proposition are exact (on $\hrt$), the conservativity of $\{E^P:\ P\in \cu'_{w=0}\}$ is 
 fulfilled if and only if for any non-zero $N\in \cu_{t=0}$ there exists $P\in \cu'_{w=0}$ such that $E^P(N)\neq \ns$.
\end{rema}


\subsection{On the existence of orthogonal weight structures}\label{sconstrw}

Now we study the question which weight structures are adjacent to $t$-structures; yet in certain cases we are only able to construct an orthogonal weight structure on a subcategory of the corresponding category.


\begin{pr}\label{ptwfromt1}
Assume that $\cu'\subset \cu$, $w$  is a weight structure on $\cu'$ that is 
 orthogonal to a  $t$-structure $t$ on $\cu$, and 
  $\cu_{t=0}\subset \obj \cu'$.

1. Then there are enough projectives in $\hrt$, 
 for any $M\in  \cu_{t=0}$ there exists an $\hrt$-epimorphism from 
 $H_0^t(P)$ into $ M$ for some $P\in \cu'_{w=0}$,   and the functor $H_0^t$  induces  an equivalence of $\kar(\hw)$
  with the category of projective objects of $\hrt$.  

2. Moreover, $\hw$ is equivalent to the latter category whenever the category $\hw$ is Karoubian. In particular, this is the case if
the class $\obj\cu'$ is retraction-closed in $\cu$ and $\cu$ is Karoubian.

3. Assume in addition that $t$ is left non-degenerate. Then $t$ is  $+$-orthogonal to $w$; hence $\cu_{t\le 0}$ is closed with respect to $\cu$-products. 
\end{pr}
\begin{proof}

1. 
Fix $M\in \cu_{t=0}$ and consider its 
$w$-decomposition $$P\stackrel{p}{\to} M\to w_{\ge 1}M\to P[1].$$ Since $M\in \cu_{t= 0}$, Proposition \ref{phw2}(1) implies that $P$ belongs to $\cu'_{w\ge 0}$; hence $P $ belongs to $ \cu'_{w=0}$ according to Proposition \ref{pbw}(\ref{iwd0})). Next,  $P\in  \cu'_{w\ge 0}\subset \cu_{t\ge 0}$ (see Proposition \ref{phw2}(1) once again); hence  the object  $P_0=t_{\le 0}P$  equals  $H_0^t(P)$ (see Proposition  \ref{prtst}(\ref{itho}). Therefore $P_0$ is projective in $\hrt$ according to Proposition \ref{phw1}(2).

Next, the adjunction property for the functor $t_{\le 0}$ (see Proposition  \ref{prtst}(\ref{itcan})) implies that $p$ factors through the $t$-decomposition 
morphism $P\to P_0$.  Now we check that the corresponding morphism $P_0\to M$ is an $\hrt$-epimorphism. This is clearly fulfilled if and only if 
$C=\co(P_0\to M)$ belongs to $\cu_{t\ge 1}$. The octahedral axiom of triangulated categories gives a distinguished triangle $(t_{\ge 1}P)[1]\to w_{\ge 1}M\to C\to (t_{\ge 1}P)[2]$; it yields the assertion in question since $w_{\ge 1}M\in \cu'_{w\ge 1}\subset \cu_{t\ge 1}$ 
 and the class $\cu_{t\ge 1}$ is extension-closed (see Proposition \ref{prtst}(\ref{itperp})). Thus we obtain that $\hrt$ has enough projectives.

Now, the category of projective objects of $\hrt$ is clearly Karoubian. As we have just verified,   for any projective object $Q$ of $\hrt$ there exists an $\hrt$-epimorphism $H_0^t(S)\to Q$ for some $S\in \cu'_{w=0}$. Since $H_0^t(S)$ is projective in $\hrt$ according  to Proposition \ref{phw2}(2), this epimorphism splits, i.e.,   $Q$ equals the image of some idempotent endomorphism of $H_0^t(S)$. Applying Proposition \ref{phw2}(2) once again and lifting this endomorphism to $\hw$ we obtain that $\kar(\hw)$ is equivalent to the category of projective objects of $\hrt$ indeed. 

2. 
Since $\hw$ is Karoubian, $\hw\cong \kar(\hw)$; hence we obtain the equivalence in question according to assertion 1.

 Next,  $\cu'_{w=0}$ is  retraction-closed in $\cu'$ since  $\cu'_{w\le 0}$ and $\cu'_{w\ge 0}$ are. 
 Consequently, if $\obj\cu'$ is retraction-closed in $\cu$ and $\cu$ is Karoubian then $\hw$ is Karoubian as well, and this concludes the proof.

3. If $t$ is  $+$-orthogonal to $w$ then $\cu_{t\le 0}$ is closed with respect to $\cu$-products according to Proposition \ref{phw2}(2). 

Now we apply Proposition \ref{phw2}(3) and obtain that it remains to verify the following: the functors of the form $E^M$ for $M\in \cu'_{w=0}$ give a conservative family of functors $\hrt\to \ab$. Now, for any object $N$ of $\hrt$ our assumptions give the existence of a projective object $P_0$ of $\hrt$ that surjects onto it. Moreover, applying Proposition \ref{phw1}(2) we obtain the existence of $P\in \cu'_{w=0}$ along with a morphism $h$ from $P$ 
  such that $H_0^t(h)$ is isomorphic to this surjection $P_0\to N$. Hence $E^P(N)\neq 0$ if $N$ is non-zero, and we obtain the conservativity in question (see Remark \ref{rcons}). 
\end{proof}

\begin{theo}\label{twfromt}
Assume that $t$ is a $t$-structure on $\cu$ and there are enough projectives in $\hrt$.

I. Assume that for any projective object $P$ of $\hrt$ there exists $P'\in P_t$ (see Definition \ref{dpmort}(2)) along with an  
$\hrt$-epimorphism $H_0^t(P')\to P$.

1. Then the full subcategory $\cu_+$ of $\cu$ whose object class equals $\cup_{i\in \z}\cu_{t\ge i}$ is triangulated. 
 Moreover, there exists a weight structure $w_+$ on $\cu_+$  
such that  $\cu_{+,w_+\ge 0}=\cu_{t\ge 0}$, and $t$ is $-$-orthogonal to $w_+$. 

 2. 
Furthermore, one can 
 extend  (see Definition \ref{dwso}(\ref{idrest})) $w_+$ as above to a weight structure $w$ on $\cu$ that  is (left) adjacent to $t$ 
 whenever any of the following additional assumptions is fulfilled:

a.  $t$ is bounded below (see Definition \ref{dtstro}(3)).  
 
b. There exists 
an integer $n$ such that $\cu_{t\le 0}\perp \cu_{t\ge n}$. 

3. Assume that $\cu$  is smashing. Then there also  exists a smashing weight structure $\wu$ on the localizing subcategory $\ccu$ of $\cu$ that is generated by $\cu_{t\le 0}$ such that  $\cu_{\wu\le 0}=\cu_{t\le 0}$, and $\hwu$ is equivalent to the subcategory of  projective objects of $\hrt$.
 
II. Assume in addition that $\cu$ satisfies the dual Brown representability property (see Definition \ref{dcomp}(\ref{idbrown})), $t$ is cosmashing and $\hrt$ has enough projectives. Then the category $\cu$ is smashing and  for any projective object $P$ of $\hrt$ there exists $P'\in P_t$ (see Definition \ref{dpmort}(2))  such that $H_0^t(P')\cong P$.

Consequently,  there exists a smashing weight structure $\wu$ on the subcategory $\ccu$ of $\cu$ mentioned in assertion I.3, such that  $\cu_{\wu\le 0}=\cu_{t\le 0}$, and $\hwu$ is equivalent to the subcategory of  projective objects of $\hrt$.


\end{theo}
\begin{proof}

I.1.  
$\cu_+$ is  triangulated since 
since it is obviously shift-stable and all $\cu_{t\ge i}$  are extension-closed.

 Next we take $W_1=\perpp \cu_{t\ge 1}\cap \obj \cu_+$, $W_2=\cu_{t\ge 0}$, and prove that $(W_1,W_2)$ is a weight structure on $\cu_+$ (cf. 
 Theorem \ref{trefl}(I)).

The only non-trivial axiom check here is the existence of $w_+$-decompositions for all objects of $\cu$. 
 Let us verify the existence of a $w$-decomposition for any $M\in \cu_{t\ge i}$ by downward induction on $i$. The statement is obvious for $i> 0$ since $M\in \cu_{+,w_+\ge 1}=\cu_{t\ge 1}$ and we can take a 'trivial' weight decomposition $0\to M\to M\to 0$. 

Now assume that existence of $w_+$-decompositions is known for any $M\in \cu_{t\ge j}$ for some $j\in \z$. We should verify the existence of weight decomposition of an element $N$ of $\cu_{t\ge j-1}$. Clearly, $N$ is an extension of $N'[-j-1]=H_0^t(N[j+1])[-j-1]$   by $t_{\ge j}N$ (see Proposition  \ref{prtst}(\ref{itho}) for the notation). Since the latter object possesses a weight decomposition, Proposition \ref{pstar}(I) (with $A=W_1$ and $B=W_2[1]$) allows us to verify the existence of a weight decomposition of $N'[-j-1]$ (instead of $N$). Our assumptions imply that there exists an epimorphism $H_0^t(P)\to N'$ with $P\in P'_t=W_1\cap W_2$. 
 Then a cone $C $ of the corresponding composed morphism  $P\to N'$ is easily seen to belong to $\cu'_{t\ge 1}$. Since both $P$ and $C$ possess weight decompositions, applying Proposition \ref{pstar}(I) once again we obtain the statement in question.

 Lastly, $t$ is $-$-orthogonal to $w$ immediately from  Proposition  \ref{prtst}(\ref{itperp}).

2. If 
  assumption a is fulfilled then we can just take $w=w_+$ since  $\cu$ obviously equals $\cu_+$.

Now suppose that assumption b is fulfilled. Similarly to the previous proof, 
  it suffices to verify that for the couple 
$w=
 (\perpp \cu_{t\ge 1},\cu_{t\ge 0})$ the corresponding 
  decomposition triangles exist for all objects of $\cu$.

 Since $w$ is an extension of $w_+$, assertion I.1 
  gives the existence of $w$-decompositions for all elements of $\cu_{t\ge 2-n}$. Next, our  orthogonality assumption on $t$ yields that $\cu_{t\le 1-n}\subset \cu_{w\le 0}$; hence one can take trivial 
 $w$-decompositions for elements of  $\cu_{t\le 1-n}$. It remains to note that $\obj \cu=\cu_{t\ge 2-n}\star \cu_{t\le 1-n}=\obj \cu$ 
   by axiom (iv) of $t$-structures, and apply Proposition \ref{pstar}(I) once again.

3. We argue similarly to the proof of part I.1 and 
  verify that  the 
   couple  $(\tilde W_1=\perpp \cu_{t\ge 1}\cap \obj \ccu,\tilde W_2= \cu_{t\ge 0})$ gives a weight structure $\wu$ on $\ccu$. 
  Indeed, this weight structure would be smashing since the class  $W_2=\cu_{t\ge 0}$ is closed with respect to $\cu$-coproducts; see Proposition \ref{prtst}(\ref{itcopr}).
   
   Once again, for this purpose it suffices to verify that the class $\tilde{C}=\tilde{C}_1\star \tilde{C}_2[1]$ equals $\obj \ccu$. Immediately from assertion I.1, $\tilde{C}$ contains $\cu_{t\ge j}$ for all $j\in \z$. Moreover, $\tilde{C}$ is extension-closed and closed with respect to $\cu$-coproducts according to Proposition \ref{pstar}(I, II.1); hence $\tilde{C}$ equals $\obj \ccu$ indeed.

Lastly, $\obj\ccu$ is retraction-closed in $\cu$ and $\cu$ is Karoubian according to  Proposition \ref{pstar}(II.2); hence $\hwu$ is
 equivalent to the subcategory of  projective objects of $\hrt$ according to Proposition \ref{phw2}(2). 


II. $\cu$ is smashing according to  Proposition 8.4.6 of  \cite{neebook} (applied in the dual form). %
Applying Proposition \ref{phw2}(2--3) we obtain that $\cu$ and $t$ satisfy the assumptions of assertion I.3.

\end{proof}

\begin{rema}\label{rtfindim}
1. Clearly, parts I.3 and II of our theorem become more interesting in the case $\ccu=\cu$. 
We will discuss certain assumptions that ensure this equality in Proposition \ref{psetgen} below.

2. Moreover, 
 Proposition \ref{ptwfromt1} along with Theorem  \ref{twfromt}(I.3,II)  can be considered as a certain complement to Theorem \ref{tsmash}(I). So we obtain that the class of  $t$-structures 
   adjacent to smashing weight ones is closely related to the one of cosmashing $t$-structures such that $\hrt$ has enough projectives.

3. The condition $\cu_{t\le 0}\perp \cu_{t\ge n}$ for $n\gg 0$ (see part I.2 of our theorem) is a natural generalization of the finiteness of the cohomological dimension condition (for an abelian category).
\end{rema}

\begin{pr}\label{psetgen}
Let $\cu$ be a smashing category generated by some set of its objects.\footnote{Actually, if $\cu$ is generated by a set $\{C_i\}$ then it is also generated by the single object $\coprod C_i$. This observation also extends to all generation assumptions of this proposition; see Corollary 2.1.3(2) of \cite{bsnew}.} 

1. Assume that $w$ is a smashing weight structure on $\cu$.  Then there exists a set $S$ of objects of $\cu$ such that $\cu_{w\ge 0}$ equals the smallest  class of  objects of $\cu$ that is closed with respect to extensions and coproducts and contains $S$.

2. Assume that $\cu$ is {\it well generated} in the sense of \cite[Remark 8.1.7]{neebook} (and \S6.3 of \cite{krauloc}),   
 $t$ is a right non-degenerate 
  $t$-structure on $\cu$, and   there exists a 
set $S$ of objects of $\cu$ such that $\cu_{t\ge 0}$ equals the smallest  class of  objects of $\cu$ that is closed with respect to extensions and coproducts and contains $S$.

Then the localizing subcategory $\ccu$ of $\cu$ that is generated by $\cu_{t\le 0}$ is $\cu$ itself.
\end{pr}
\begin{proof}
1. Immediate from Proposition 2.3.2(10) of \cite{bwcp}.

2. Since $t$ is right non-degenerate and $\cap_{i\in \z}\cu_{t\le i}=\cap_{i\in \z}(\cu_{t\ge i}\perpp)$ (see Proposition \ref{prtst}(\ref{itperp})), we obtain  $\ccu^{\perp_{\cu}}=\ns$. 

Next, assume that the embedding $\ccu\to \cu$ possesses a right adjoint. Then this functor is the localization by the subcategory $\ccu^{\perp_{\cu}}$; see Propositions 4.9.1 and 4.10.1 of \cite{krauloc}. 

Consequently, it remains to verify that  the embedding $\ccu\to \cu$ possesses a right adjoint indeed. Since $\ccu$ is generated by $S$ as a localizing subcategory of $\cu$, the latter statement is a well-known combination of Theorems 7.2.1(2) and 5.1.1(2) of loc. cit.

\end{proof}

\begin{rema}\label{rcasa}
1. Consequently, if  $\cu$ is  well generated, satisfies the dual Brown representability property, and $t$  is a right non-degenerate cosmashing $t$-structure on $\cu$ then there exists a weight structure 
 adjacent to $t$ if and only if $\hrt$ contains  enough projectives and there exists a set $S$ that generates  $\cu_{t\ge 0}$ in the sense specified in Proposition \ref{psetgen}(2); see Proposition \ref{ptwfromt1}(1) and Theorem \ref{twfromt}(II). 

2. Recall also that Theorem 3.9 of \cite{casa} ensures the existence of the right adjoint to  the embedding $\ccu\to \cu$ under certain assumptions that do not depend on the existence of a set of generators $S$ (cf.  Proposition \ref{psetgen}(2)).

\end{rema}

\subsection{Some adjacent  weight structures in the saturated case}\label{sconstrwsat}

Now we proceed towards extending Theorem \ref{twfromt}(I.2a) to a rather nice statement on saturated categories. To formulate it in a rather general form we need some definitions related to $t$-structures. 

\begin{defi}\label{dtstreb}

Let $t$ be a  $t$-structure  on $\cu$.

1. We will say that $t$ is {\it essentially bounded (below)} whenever for any $M\in \obj \cu$ we have $H^t_i=0$ for $|i|\gg 0$ (resp. for $i\ll 0$; see Proposition \ref{prtst}(\ref{itho})).

2. We define $\cu_{t=-\infty}$ as the  full subcategory of $\cu$ whose  object class equals  $\cap_{i\in \z}\cu_{t\le i}$.

\end{defi}


\begin{pr}\label{pteb}
Let $t$ be a $t$-structure $t$ on $\cu$.

I.1. Then the subcategory 
$\cu_{t=-\infty}$ 
 is triangulated and $t$ restricts to it.

2. The restriction $t_+$ of $t$ to $\cu_+$ (see Theorem \ref{twfromt}(I.1)) is bounded below. Moreover, $t_+$ is essentially   bounded if $t$ is. 

3. $\obj \cu_+\perp \obj \cu_{t=-\infty}$.

II. Assume that $t$ is 
essentially bounded below. 

Then $(\cu_{t=-\infty}, \cu_+)$ is a semi-orthogonal decomposition of $\cu$.


\end{pr}
\begin{proof}
I. All of these statements are quite simple.

1. The subcategory 
 $\cu_{t=-\infty}$ 
  is triangulated since it is 
  obviously shift-stable and  all $\cu_{t\le i}$ 
    are extension-closed 
 (see Proposition \ref{prtst}(\ref{itperp})). 
 $t$ restricts to them since the functors 
  $L_t$ and $R_t$ respect 
 all 
  $\cu_{t\le i}$; here one can apply axiom (ii) of Definition \ref{dtstr} and  Proposition \ref{prtst}(\ref{itcan},\ref{itho}).

2,3. Obvious. 

II. We should (see Definition \ref{ddec}) verify that for any $M\in \obj \cu$ there exists a distinguished triangle \begin{equation}\label{edec} B\to M\to A\to B[1]\end{equation}
 with $B\in \obj \cu_+$ and $A\in \obj \cu_{t=-\infty}$. Since $t$ is essentially bounded below, there exists $m\in \z$ such that $H_i^t=0$ for $i<m$. Then for the corresponding distinguished triangle $$t_{\ge m}M\to M\to t_{\le m-1}M\to (t_{\ge m}M)[1]$$ (this is the shift of the $t$-decomposition triangle for $M[m]$ by $[-m]$; cf. Proposition \ref{prtst}(\ref{itho})) we have  $t_{\ge m}M\in \cu_{t\ge m}\subset  \obj \cu_+$, whereas the object $A=t_{\le m-1}M$ equals $t_{\le m-2}M=t_{\le m-3}M=\dots$; hence $A\in 
 \obj \cu_{t=-\infty}$ indeed.
 \end{proof}
 
 \begin{theo}\label{tsatw}
Assume that $\cu\opp$ is an $R$-saturated category,  $t$ is an essentially bounded below $t$-structure on $\cu$, and  there are enough projectives in $\hrt$.

1. If $t$ is also essentially bounded then for any projective object $P$ of $\hrt$ there exists $P'\in P_t$ (see Definition \ref{dpmort}(2)) such that $H_0^t(P')\cong P$.

2.  Assume that for any projective object $P$ of $\hrt$ there exists $P'\in P_t$ (see Definition \ref{dpmort}(2)) along with an  
$\hrt$-epimorphism $H_0^t(P')\to P$.

Then there exists an essentially bounded below weight structure $w$ on $\cu$ (left) adjacent to $t$. 
\end{theo}
\begin{proof}
1. It suffices to note that the functor $\hrt(P,-)\circ H_0^t$  is locally bounded (as a functor from $\cu\opp$ into $R-\modd$) and takes values in $R-\mmodd$. 

2. We construct this weight structure as an extension of the weight structure $w_+$ provided by Theorem \ref{twfromt}(I.1). 

As we have noted in Remark \ref{rsatur}(2), applying  Corollary \ref{csatur}(1) to the category $\cu\opp$ (and its semi-orthogonal decomposition $(\cu_+, \cu_{t=-\infty})$; see Proposition \ref{pteb}(2) and Remark \ref{rsod}(2)) 
 one 
  obtains a semi-orthogonal decomposition $(\cu_+,\cu_0)$ of $\cu$ 
   with $\obj \cu_0=\perpp(\obj \cu_+)$. Then Proposition 3.2(5) of \cite{bvt} implies that $w=(\obj \cu_0 \star \cu_{+,w_{+\le 0}},\cu_{t\ge 0})$ is a weight structure. Here we apply Proposition \ref{pteb}(II) and duality to obtain the existence of the adjoint functor mentioned in loc. cit.,  and recall 
 that $\cu_{t\ge 0}=\cu_{+,w_{+\ge 0}}$. 
 
 Lastly, this $w$ is clearly 
    adjacent to $t$ (see Proposition \ref{portadj}). 
 \end{proof}

\begin{rema}\label{rsaturw}
1. Clearly, all the statements and definitions of this section can be dualized.

2. It would  be interesting to argue similarly to the proof of Theorem \ref{tsatw} in the context of Theorem \ref{twfromt}(II). The main problem is to find conditions that would enforce 
$\ccu$ to be closed with respect to $\cu$-products.
\end{rema}

\section{Some additional remarks}\label{smore} 

In this section we make some (not really important) remarks. Most of them relate the results above to various statements in the literature. We also discuss some examples. 
	
To 
relate 
our Definition \ref{dort}(1) (of orthogonality between $w$ in $t$) to the more general definition given in  \cite{bger} we recall Definition 2.5.1(1) of ibid. 

\begin{ddefi}\label{ddual}
We will call a (covariant) bi-functor
  $\Phi:\cu^{op}\times\cu'\to
\au$ a {\it duality} between $\cu$ and $\cu'$ if  $\Phi$ is 
 homological with respect
to both arguments; and is equipped with a (bi)natural bi-additive transformation
$\Phi(X,Y)\cong \Phi (X[1],Y[1])$.
\end{ddefi}

\begin{rrema}\label{rldual}
\begin{enumerate}
\item\label{irort1}
Assume $\cu,\cu'\subset \du$. Then the restriction of the bi-functor $\du(-,-)$ to $\cu^{op}\times\cu'$ is easily seen to give a duality $\cu^{op}\times\cu'\to \ab$; see
 Proposition 2.5.6(1) of \cite{bger}.

\item\label{irort2}
In Definition 2.5.1(3) of \cite{bger} 
  orthogonality was defined 
  with respect to  an arbitrary duality $\Phi$ between $\cu$ and $\cu'$; see Definition \ref{ddual}.

  Consequently, all the properties of orthogonal weight and $t$-structures established in ibid. can be applied in the setting of Definition \ref{dort}(1).
  \end{enumerate}\end{rrema}
  


\begin{rrema}\label{rbvtt}
1. This paper is a certain modification of the preprint \cite{bvtr}. 
One of the main distinctions is that the somewhat ad hoc (and restrictive) notion of reflection of categories (see Definition 2.2.1 of ibid.) was avoided. 
Respectively, Corollary \ref{ccopr} is a certain substitute for Theorem 2.2.5 of \cite{bvtr}.
Though the author suspects that the results of the current text are not sufficient to prove loc. cit. itself, 
 they can be successfully applied to treat all the examples of loc. cit. known to the author; 
  cf. Proposition \ref{pgeomap1}. 

Moreover, several new general results were added; this includes 
 the theory of coproductive extensions, 
   (most of) \S\ref{ststreb},  \S\ref{sconstrwsat}, and subcategories corresponding to certain support conditions (see Propositions \ref{pcgrlin}, \ref{pcgrlintba},  \ref{pgeomap1}, and \ref{pgeomap2}(2), and Corollary \ref{csatur}(2)). 
 
 On the other hand, 
   the current text does not include Proposition 2.2.7, Corollary 4.1.4(2), and \S4.4  of ibid. 
    Probably, these statements will be included into a succeeding paper.

Note also that the terminology of the current text is substantially distinct from the one of ibid.

2. Since \cite{bvtr} (as well as the earlier preprint \cite{bpgws}) is cited in some recent papers, 
 it makes sense to note the following: Theorem 3.2.3(I)  of \cite{bvtr} (as well as Theorem 3.1.2 of \cite{bpgws} that has been cited by several authors) 
 coincides with Theorem \ref{tadjti}  of the current paper,  Theorem 3.3.1 of \cite{bvtr}  almost coincides with Theorem \ref{tcompws},   Theorem 5.3.1(I.1) of ibid. is contained in Proposition \ref{ptwfromt1}, Theorem 5.3.1(IV)  of ibid. was generalized to 
 Theorem \ref{tsatw},  and strongly $\alz$--well generated weight structures (that were treated in  \S3.3 of  \cite{bvtr}) are discussed in \S\ref{sort} (yet see Remark \ref{rrcuzkar} below).
\end{rrema}

\begin{rrema}\label{rneetgen}
The methods of the current paper are not the most general among the existing methods for constructing $t$-structures. In particular, if $\cu$ is a well generated triangulated category 
 then the recent Theorem 2.3 of \cite{neetsgen} gives
all those $t$-structures that are generated by sets of objects of $\cu$ in the sense of Definition \ref{dcomp}(\ref{dgenw}). 

Now, 
if $\cu$ is well generated then is generated by a set of its objects (as its own localizing subcategory; see Definition \ref{dcomp}(\ref{dlocal})). 
 Thus for any smashing weight structure $w$ on it Proposition 2.3.2(10) of \cite{bwcp} essentially says the following: there exists a set $\cp\subset \cu_{w\ge 0}$ such that the class $\cu_{w\ge 0}$ is the  smallest cocomplete pre-aisle that is generated by 
  $\cp$ in the sense of \cite[\S0]{neetsgen} (cf. Discussion 1.16 of ibid.). 
Hence Theorem 2.3 of \cite{neetsgen} implies that there exists a $t$-structure on $\cu$ such that $\cu_{t\ge 0}=\cu_{w\ge 0}$ (cf. Remark \ref{rtst}(2) below). Thus loc. cit. generalizes the existence of $t$ part of our Theorem \ref{tadjti}.

On the other hand, note that neither loc. cit. nor the well-known Theorem A.1 of \cite{talosa} says anything on the hearts of $t$-structures. Moreover,  there appears to be no chance to extend these existence results to $R$-saturated categories (in any way). 
\end{rrema}

\begin{rrema}\label{rexam}

Let us discuss examples to Theorem \ref{tsmash}; cf. Remark \ref{rsmashex}.

1. First we apply the theorem to  semi-orthogonal decompositions.


If $w$ is a semi-orthogonal decomposition of $\cu$ (where $\cu$ satisfies the Brown representability property) that is smashing as a weight structure then for the corresponding $t^l$ the couple $(\cu_{t^l\le 0},\cu_{t^l\ge 0})$ is a semi-orthogonal decomposition as well; see Proposition \ref{portsod}. 
  Note also that $(\cu_{t^l\le 0},\cu_{t^l\ge 0})$ is a (cosmashing) weight structure.

Next, recall that the full subcategory $\cu_1$ of $\cu$ corresponding to $C_1=\cu_{w\le 0}$ is  triangulated and  
 the corresponding right adjoint (to the embedding $\cu_1\to \cu$) respects coproducts; see condition \ref{isoda} in Proposition \ref{portsod} and Remark \ref{rsod}(3). 
 Since $w\opp$ is a semi-orthogonal decomposition in the category $\cu\opp$ (see Proposition \ref{pbw}(\ref{idual}), we obtain that  the embedding $\cu_2\to \cu$ possesses a left adjoint; here $\cu_2$ the subcategory of  $\cu$ corresponding to $\cu_{w\ge 0}$.
 
 Thus if we apply Proposition \ref{portsod} 
 to the 
  semi-orthogonal decomposition $(\cu_{t^l\le 0},\cu_{t^l\ge 0})$  we obtain that the embedding $\cu_2\to \cu$ possesses a right adjoint as well; note that $\cu_{t^l\ge 0}=\cu_{w\ge 0}$. 
Thus  $\cu_2$ is {\it admissible} in $\cu$ in the sense of  \cite[Definition 2.5]{bondkaprserre} and the embedding $\cu'\to\cu$ may be completed  to a {\it gluing datum} (cf. \cite[\S1.4]{bbd} or \cite[\S9.2]{neebook}).

So we generalize Corollary 2.4 of \cite{nisao} to arbitrary categories that satisfy the  Brown representability property. 

2. Let us say a little more on smashing weight structures.

Theorem 4.5(2) of \cite{postov} and Theorem 
 3.2.1 of \cite{bgroth} 
 enable one to check  whether two weight structures 
 generated by a set of compact objects are distinct.

Other notable statements related to the construction of smashing weight structures are Theorem 2.3.4(3) of \cite{bgroth} (it says that any {\it perfect} set of objects in a smashing triangulated category generates a smashing weight structure) and Theorem 3.1.3 and 3.2.2 of \cite{bsnew} that treat weight-exact localization functors.
\end{rrema} 

\begin{rrema}\label{rrcuzkar} 
Let us make a funny observation concerning Theorem \ref{tcompws}.

According to the well-known Lemma 4.4.5 of \cite{neebook}, the assumptions of Theorem \ref{tcompws} imply that $ \cu^{(\alz)}$ is the retraction-closure (see \S\ref{snotata}) of the category $\cuz$.  Note however that $\wz$ does not extend to $ \cu^{(\alz)}$ in general (even though it extends to $\cu$!); see \S3.1 of \cite{bonspkar}.\footnote{In any example of this sort $\wz$ is neither bounded above nor below; see Theorem 2.2(II.2) of ibid.} Consequently, it does makes some sense not to 
 assume $\cuz= \cu^{(\alz)}$ in our theorem.

 For the same reason, the weight structures 
  provided by our Theorem \ref{tcompws} don't have to be {\it strongly $\alz$--well generated} in the sense of Remark 3.3.4(1) of \cite{bgroth} (even if one assumes that $\cuz$ is essentially small; note that this is equivalent to the compact generation of $\cu$). This observation possibly suggests that the notion  of strong $\alz$--well generation should be slightly generalized.
\end{rrema}

\end{document}